\documentclass[final]{siamart1116}

\usepackage[utf8]{inputenc}
\usepackage[T1]{fontenc}

\newcommand{\TheTitle}{The Effect of Crystal Symmetries on the Locality of Screw Dislocation Cores}
\newcommand{\shortTitle}{Locality of Screw Dislocation Cores}
\newcommand{\TheAuthors}{Julian Braun, Maciej Buze, and Christoph Ortner}

\headers{\shortTitle}{\TheAuthors}

\title{{\TheTitle}\thanks{\funding{MB is supported by EPSRC as part of the MASDOC DTC, Grant No. EP/HO23364/1. JB and CO are supported by ERC Starting Grant 335120.}}}

\author{
  Julian Braun\thanks{Mathematics Institute, University of Warwick, Coventry, CV4 7AL, UK.}
  \and
  Maciej Buze\footnotemark[2]
  \and
  Christoph Ortner\footnotemark[2]
}

\ifpdf
\hypersetup{
  pdftitle={\TheTitle},
  pdfauthor={\TheAuthors}
}
\fi

\usepackage{amsmath}
\usepackage{amssymb}
\usepackage{amsopn}
\usepackage{enumitem}
\usepackage{color}
\usepackage{mathtools}
\usepackage{hyperref}
\usepackage{doi}
\usepackage[normalem]{ulem}

\numberwithin{theorem}{section}
\newsiamremark{remark}{Remark}
\newsiamremark{example}{Example}
\newsiamremark{assumption}{Assumption}

\usepackage{cleveref}
\Crefname{subsection}{Section}{Sections}

\DeclareMathOperator{\modulo}{mod}

\DeclareMathOperator{\tr}{tr}

\DeclareMathOperator{\divo}{div}
\DeclareMathOperator{\divRc}{Div}   
\def\DRc{D}                         
\DeclareMathOperator{\Id}{Id}
\DeclareMathOperator{\sym}{sym}

\DeclareMathOperator{\spano}{span}

\DeclareMathOperator{\spt}{spt}

\DeclareMathOperator{\Oo}{O}

\newcommand{\sfrac}[2]{{\textstyle \frac{#1}{#2}}}

\newcommand{\R}{\mathbb{R}}

\newcommand{\Rc}{\ensuremath{\mathcal{R}}}

\newcommand{\Hcc}{\ensuremath{\dot{\mathcal{H}}^1}}
\newcommand{\La}{\Lambda}

\newcommand{\Hhom}{H}  
\newcommand{\<}{\langle}
\renewcommand{\>}{\rangle}

\newcommand{\Z}{\mathbb{Z}}
\newcommand{\C}{\mathbb{C}}
\newcommand{\N}{\mathbb{N}}

\newcommand{\eps}{\varepsilon}

\usepackage{color}

\begin{document}

\maketitle
\begin{abstract}
  In linearised continuum elasticity, the elastic strain due to a straight
  dislocation line decays as $O(r^{-1})$, where $r$ denotes the distance to the
  defect core. It is shown in \cite{EOS2016} that the {\em core correction} due to nonlinear and discrete (atomistic) effects decays like $O(r^{-2})$.

  In the present work, we focus on screw dislocations under pure anti-plane
  shear kinematics. In this setting we demonstrate that an improved decay
  $O(r^{-p})$, $p > 2$, of the core correction is obtained when crystalline
  symmetries are fully exploited and possibly a simple and explicit correction
  of the continuum far-field prediction is made.

  This result is interesting in its own right as it demonstrates that, in some
  cases, continuum elasticity gives a much better prediction of the elastic
  field surrounding a dislocation than expected, and moreover has practical
  implications for atomistic simulation of dislocations cores, which we discuss
  as well.
\end{abstract}



\section{Introduction}
\label{sec:intro}
Crystalline solids consist of regions of periodic atom arrangements, which are
broken by various types of defects. Crystalline defects can be separated into an
elastic far-field which can normally be described by continuum linearised
elasticity (CLE) and a defect core which is inherently atomistic and determines,
for example, mobility, formation energy (and hence concentration), and so forth.

\def\up{u_{\rm ff}}
\def\uc{u_{\rm core}}
\def\L{\Lambda}

To make this idea concrete, let $\L \subset \R^d$ be a crystalline lattice
reference configuration and let $u : \L \to \R^d$ be an equilibrium displacement field under some interaction law (see \S~\ref{model}). The point of view advanced in \cite{EOS2016} is to decompose $u =\up + \uc$ where $\up$ is a {\em far-field predictor} solving a CLE equation enforcing the presence of the defect of interest and $\uc$ is a {\em core corrector}.  For example, it is shown in \cite{EOS2016} that for dislocations $|D \up(x)| \sim |x|^{-1}$ while $|D \uc(x)| \lesssim |x|^{-2} \log |x|$ where $D$ denotes a discrete gradient operator. The fast decay of the corrector
$\uc$ encodes the ``locality'' of the defect core (relative to the far-field).

The present work is the first in a series that introduces and developes
techniques to substantially improve on the CLE far-field description. The
overarching goal is to derive ``higher-order'' models for the
far-field predictor $\up$, which yield the same asymptotic behaviour as the CLE
predictor (i.e., the same far-field boundary condition) but a more localised
corrector. For example, in the case of a dislocation we seek $\up$ such that $u =
\up + \uc$ with $|D \uc(x)| \lesssim |x|^{-p}$ and $p > 2$. Constructions of
this kind have a multitude of applications. They are interesting in their own
right in that they give improved estimates on the region of validity of
continuum mechanics. They may also be employed to more effectively construct
models for multiple defects along the lines of \cite{HO2014}. A key motivation
for us is that they yield a new class of boundary conditions for atomistic
simulations that capture the far-field behaviour more accurately; this gives
rise to improved algorithms for atomistic simulation defects; see
\S~\ref{numerics} for more detail.

In the present work, to demonstrate the potential of our approach and outline
some of the key ideas required to carry out this programme, we focus on screw
dislocations under anti-plane shear kinematics, in the cubic, hexagonal, and
body-centred-cubic (BCC) lattices. The scalar setting, and the ability to
exploit specific lattice symmetries, simplifies several constructions and
proofs.

In forthcoming papers, in particular \cite{MainPaper}, we will discuss
generalisations to vectorial deformations of general straight dislocations
without any symmetry assumptions on the host crystal. In particular the absence
of the symmetries we employ in the present work introduces a non-trival coupling
between the core and the far-field predictor. The general idea that persists is
that there is a development $u = u_0 + u_1 + \cdots + u_n + u_{\rm rem}$ of the
solution, where the terms $u_0, u_1, \dots, u_n$ are given by simpler theories (e.g., linear PDEs) and the remainder $u_{\rm rem}$ has a higher decay rate.

Aside from providing a simplified introduction to \cite{MainPaper}, the present
work contains results that are interesting in their own right due to the fact
that anti-plane models of screw dislocations are particular popular in the
mathematical analysis literature \cite{AdLGP2014, Hudson2017, HO2014, Ponsiglione2007} as a model problem for the more complex edge, mixed, and curved
dislocations. Of particular note about our results here are:

(1) A combination of rotational and anti-plane reflection symmetries yield
surprisingly high decay of the core corrector to the CLE predictor; see
Theorem~\ref{thm:highsym}. This was numerically observed but unexplained in
\cite{HO2014}. The key observation to obtain this result is that the CLE
predictor satisfies additional PDEs, in particular the minimal surface equation,
which naturally occurs in higher-order expansions of the atomistic forces.

(2) In a BCC crystal, due to the lack of anti-plane reflection symmetry, a
nonlinear correction to the far-field predictor is required to improve the decay
of the core corrector. One then expects that the dominant error contribution is
the Cauchy--Born anti-discretisation error. The results of \cite{braun16static, emingstatic, ortnertheil13}  suggest that the resultant corrector should decay
as $O(\lvert x \rvert^{-3})$, however exploiting crystal symmetries reveals that the
Cauchy--Born error is of higher order than expected and one even obtains a
corrector decay of $O(\lvert x \rvert^{-4})$.

Finally, we remark that our analysis is carried out for short-ranged interatomic
many-body potentials, however the resulting algorithms are applicable to
electronic structure models rendering them an efficient and attractive
alternative to complex and computationally expensive multi-scale schemes e.g, of
atomistic/continuum or QM/MM type; see  \cite{2015-qmtb2, 2013-atc.acta} and
references therein.

{\bf Outline: } In \Cref{sec:mainresults} we describe in details our models and
assumptions, and state our main results. Here, \Cref{symsec} is dedicated to the
cubic and hexagonal lattice, while \Cref{BCC} discusses the BCC lattice. In
\Cref{numerics} we present the resulting new numerical scheme including a
convergence analysis. Our conclusions can be found in \Cref{sec:Conclusion}. Finally \Cref{Proofs} contains the proofs of the main results.

\section{Main results}
\label{sec:mainresults}

\subsection{Atomistic model for a screw dislocation}\label{model}
The atomistic reference configuration for a straight screw dislocation is given by a two-dimensional Bravais lattice $\La = A_\La \Z^2$, $A_\La \in \R^{2 \times 2}$ with $\det(A_\La) \neq 0$. In the present work we will only consider the triangular lattice and the square lattice, respectively given by
\[
 A_\La =A_{\rm tri}:=
  \begin{pmatrix}
  1 & \frac{1}{2} \\
  0 & \frac{\sqrt{3}}{2}
 \end{pmatrix}, \qquad A_\La =A_{\rm quad}:=
  \begin{pmatrix}
  1 & 0 \\
  0 & 1
 \end{pmatrix}.
\]
The two-dimensional lattice $\La$ should be thought of as the projection of a three-dimensional lattice: In case of an infinite straight dislocation in a three-dimensional lattice, the displacements do not depend on the dislocation line direction. Therefore, it suffices to consider the projected two-dimensional lattice.

Our atomistic model, which we specify momentarily, allows for general finite
range interactions. All lattice directions included in the interaction range are
encoded in a finite neighbourhood set $\Rc \subset \La \backslash \{0\}$, which
is fixed throughout. We always assume $\spano_\Z \Rc =\La$ and will specify
further symmetry assumptions later on.

We consider an anti-plane displacement field $u\, :\,\Lambda \to \R$ and define $D_\rho u(x):=u(x+\rho)-u(x)$,  $\DRc u(x) := (D_\rho u(x))_{\rho \in \Rc}$, as well as the discrete divergence operator,
\[
 \divRc g(x) := -\sum_{\rho \in \Rc} g_{\rho}(x-\rho) - g_{\rho}(x)
 \qquad \text{for any } g\colon \Lambda \to \R^{\Rc}.
\]
In contrast to that we will always write $\nabla$ and $\divo$ if we talk about the standard (continuum) gradient and divergence of differentiable maps.

A suitable function space for (relative) displacements is
\[
\Hcc := \{ u\, :\,\Lambda \to \R\, | \, \DRc u \in \ell^2(\La)\} /\R,
\]
with norm
\[
 \|u\|_{\Hcc} := \Big(\sum_{x \in \La} \big\lvert \DRc u(x) \big\rvert^2\Big)^{1/2}.
\]
While we have factored out constants to make this a Banach space, we will often use the displacement $u$ and its equivalence class $[u]$ interchangeably when there is no risk of confusion.

For analytical purposes, we will also consider the space of compactly supported displacements
\[
\mathcal{H}^c := \{ u\, :\,\Lambda \to \R \, | \, \spt(u) \text{ is bounded}\}.
\]

Displacement fields containing dislocations do not belong to $\Hcc$ and the
energy, naively written as a sum of local contributions, will be infinite.
Following \cite{EOS2016,HO2014} we therefore consider energy differences
\begin{equation}\label{energy} \mathcal{E}(u) = \sum_{x\in\La}\Big(V(\DRc
\hat{u}(x) + \DRc u(x)) - V(\DRc \hat{u}(x))\Big), \end{equation} where
$\hat{u}$ is a chosen {\em far-field predictor} that encodes the far-field
boundary condition, while $u \in \Hcc$ is a {\em core corrector} to the given
predictor so that $\hat{u} +u$ gives the overall displacement. We will minimise
$\mathcal{E}(u)$ to equilibrate the defective crystal, but this requires some
preparation first.

We assume throughout that $V \in C^6(\R^\Rc, \R)$ is a many-body potential encoding the local interactions. Examples of typical site potentials $V$ include Lennard-Jones type pair potentials (with cut-off) and EAM potentials; see also \Cref{BCC} and \Cref{numerics}. With additional effort it would be possible to include simple quantum chemistry models within the framework \cite{ChenOrtner2016a}, but to keep the presentation as simple as possible we will not pursue this in the current work.

As $\Lambda$ is either the square or triangular lattice, which are both invariant under certain symmetries, one is tempted to directly translate these symmetries to $\Rc$ and $V$. However, as mentioned above, $\Lambda$ should be seen as a projection of a three-dimensional lattice. Such a projection can add symmetries for the lattice that are not reflected in the interaction, since they are not symmetries of the underlying three-dimensional model. We will discuss such a case in detail in \Cref{BCC}.

Because of this, we will only make the following reduced symmetry assumptions on
$\Rc$ and $V$ throughout. Let $Q_\Lambda$ be the rotation by $\pi/2$ if $\Lambda =
\Z^2$ and the rotation by $2\pi/3$ if $\Lambda = A_{\rm tri} \Z^2$. Then we assume
that
\begin{equation}\label{eq:rotsym}
  Q_\La \Rc =\Rc \qquad \text{and} \qquad V(A) = V\left((A_{Q_\La \rho})_{\rho \in \Rc}\right) \qquad \forall A \in \R^\Rc.
\end{equation}

Since we only consider a plane orthogonal to the direction of the dislocation line, it is natural that the energy does not change if the displacements shifts an atom to its equivalent position in the plane above or below. Indeed we assume that there is a minimal periodicity $p>0$ such that
\[
  V(A) = V(A + p(\delta_{\rho \sigma})_{\sigma \in \Rc})
  \qquad \text{for all } A \in \R^\Rc, \quad \rho \in \Rc.
\]
The Burgers vector of a screw dislocation is then either $b=p$ or $b=-p$.

A key conceptual assumption that we require throughout this work is
lattice stability (or, phonon stability): there exists $c_0 > 0$
such that
\begin{equation} \label{eq:latticestability}
	\< \Hhom u, u \> \geq c_0  \|u \|_{\Hcc}^2 \qquad \forall u \in \Hcc,
\end{equation}
where $\Hhom$ denote the Hessian of the potential energy evaluated at the
homogeneous lattice (note that this is different from $\delta^2 \mathcal{E}(0)$),
\begin{equation*}
	\< \Hhom u, v \> = \sum_{x \in \La}
		\sum_{\rho, \sigma \in \Rc}
			\nabla^2V(0)_{\rho\sigma} D_\rho u(x) D_\sigma v(x),
	\qquad \text{for } u, v \in \Hcc.
\end{equation*}

The choice of the predictor $\hat{u}$ in \cref{energy} for a specific problem
is part of the modelling since it determines the far-field behaviour, e.g.,
it could encode an applied strain. Intuitively one can obtain a suitable $\hat{u}$ by
solving a ``simpler'' model such as continuum linearised elasticity (CLE), which
one expects to be approximately valid in the far-field; see \cite{EOS2016} for
a formalisation of this procedure.

Assume, for the time being, that $\hat{u} : \R^2 \to \R$ is smooth away from a
defect core $\hat{x} \in \R^2 \setminus \La$. Then, by employing Taylor expansions of both $\hat{u}$
and of $V$, we can approximate the atomistic force, \begin{align} \notag
\frac{\partial \mathcal{E}}{\partial u(x)} \Big|_{u = 0}
	&= \Big(-\divRc \nabla V(\DRc\hat{u})\Big)(x)
   \\ &= \label{eq:formalforceexpansion}
	-c\divo \nabla W(\nabla\hat{u}) + O(\nabla^4\hat{u}(x)) + \text{h.o.t.s}
	\\ &= \notag
	-c\divo (\nabla^2 W(0)[\nabla\hat{u}])
		+ O(\nabla^4\hat{u}) + O(\nabla^2\hat{u}\nabla\hat{u})
			+ \text{h.o.t.s},
\end{align}
where $W\,: \R^2 \to \R$ is the
\textit{Cauchy-Born energy per unit undeformed volume}, defined by
\begin{equation}\label{CB-W}
W(F):= \frac{1}{\det A_{\La}}V(F\cdot\Rc).
\end{equation}
Moreover,  $O(\nabla^4\hat{u})$ represents the {\em anti-discretisation error}
(note that the continuum model is now the approximation),
$O(\nabla^2\hat{u}\nabla\hat{u})$ the linearisation error and ``h.o.t.s''
denotes additional terms that will be negligible in comparison.

It is therefore natural to solve a CLE model to obtain a far-field predictor
for the atomistic defect equilibration problem for an anti-plane screw
dislocation. Let $\hat{x} \in \R^2$ denote the dislocation core, then we
define the branch cut (slip plane)
\[
 \Gamma := \big\{(x_1,\hat{x}_2)\in \R^2\;|\; x_1 \geq \hat{x}_1\big\}
\]
and solve (we will see in \Cref{D2W-cor} that under our general assumptions on $\Rc$ and $V$ we have $\nabla^2W(0) \propto \Id$)
\begin{subequations} \begin{align}
 - \Delta \hat{u} &= 0 \quad \text{ in }\, \R^2 \setminus \Gamma,\label{e1a}\\
 \hat{u}(x^+) - \hat{u}(x^-) &= -b\, \quad \text{ on }\,\Gamma\setminus\hat{x},\label{e2a}\\
 \partial_{x_2}\hat{u}(x^+) - \partial_{x_2}\hat{u}(x^-) &= 0\, \quad\text{ on }\,\Gamma\setminus\hat{x}\label{e3a}.
\end{align} \end{subequations}
The system \cref{e1a}--\cref{e3a} has the well-known solution (cf. \cite{hirth-lothe})
\begin{equation}\label{screwpredictorformula}
	 \hat{u}(x)= \frac{b}{2\pi}\arg(x-\hat{x}),
\end{equation}
where we identify $\R^2 \cong \C$ and use $\Gamma - \hat{x}$ as the branch cut
for $\arg$. Note for later use, that $\nabla \hat{u} \in C^\infty(\R^2\backslash \{0\})$ and $\lvert \nabla^j \hat{u} \rvert \lesssim \lvert x \rvert^{-j}$ for all $j \geq 0$ and $x \neq 0$.

As we want to study the effects of symmetry, we will assume throughout that the dislocation core $\hat{x}$ is, respectively, at the center of a triangle or square.

Having specified the far-field predictor we can now recall properties of the
resulting variational problem.

\begin{proposition} \label{energysmooth}
	Let $\hat{u}$ be given by \cref{screwpredictorformula}, then
	$\mathcal{E}$ defined by \cref{energy} on $\mathcal{H}^c$ has a unique continuous extension $\mathcal{E} \colon \Hcc \to \R$. Furthermore, $\mathcal{E} \in C^6(\Hcc)$.
\end{proposition}
\begin{proof}
This is proven in \cite[Lemma 3 and Remark 6]{EOS2016}.
\end{proof}

Having established that $\mathcal{E}$ is well-defined, it is now meaningful
to discuss the equilibration problem, either energy minimisers
\begin{equation} \label{eq:minproblem}
	\bar{u} \in \arg\min_{\Hcc} \mathcal{E}
\end{equation}
or, more generally, critical points
\begin{equation*}
	\delta\mathcal{E}(\bar{u}) = 0.
\end{equation*}
Critical points of the energy satisfy the following regularity and decay estimate.

\begin{theorem} \label{EOSdecay}
If $[\bar{u}] \in \Hcc$ is a critical point of $\mathcal{E}$, then there exists $\bar{u}_{\infty} \in \R$ such that
\[
	\lvert D^j(\bar{u}(x) - \bar{u}_{\infty}) \rvert
		\lesssim   \lvert x \rvert^{-j-1} \log \lvert x \rvert,
\]
for all $\lvert x \rvert$ large enough and $0 \leq j \leq 4$.
\end{theorem}
\begin{proof}
This result is proven in \cite[Theorem 5 and Remark 9]{EOS2016}.
\end{proof}

\subsection{Anti-plane screw dislocations with mirror symmetry}\label{symsec}
The corrector decay rates in \cite{EOS2016} are in general sharp (up to
constants and log-factors), however the case of anti-plane screw dislocations
appears to be an exception: In  \cite{HO2014} it is seen numerically for a
triangular lattice that, if the core is placed at the centre of a triangle, one
approximately has $\lvert Du(x)\rvert \sim \lvert x \rvert^{-4}$
instead of the expected rate $\lvert x \rvert^{-2} \log \lvert x \rvert$. In the
present section we relate this observation to several symmetry properties of the
triangular lattice. We also discuss the square lattice case which shows a
different behaviour to emphasise the importance of the triangular lattice.

These two-dimensional models represent a screw-dislocation in a cubic or
hexagonal three-dimensional lattice only allowing for anti-plane displacements.
In \Cref{BCC}, we will additionally consider a BCC lattice and show how to
derive these two-dimensional systems from the underlying three-dimensional
model.

We recall that $\Lambda$ is either the square or triangular lattice which are
both invariant under certain rotational symmetries. Crucially, we consider
rotations about the dislocation core (not about a lattice site), which are
described by the operators
\begin{equation*}
	L_\Lambda x := Q_\Lambda (x - \hat{x}) + \hat{x},
\end{equation*}
where $Q_\Lambda$ denotes a rotation through $\pi/2$ if $\Lambda = \Z^2$ and a
rotation through $2\pi/3$ if $\Lambda = A_{\rm tri} \Z^2$. Since we assumed that
$\hat{x}$ lies, respectively, at a center of triangle or square this implies
$L_\La \La = \La$.

In the present section we additionally assume mirror symmetry with respect to
the plane orthogonal to the dislocation line, which is encoded in the site
energy through the assumption
\begin{equation} \label{eq:mirrorsymV}
  V(A) = V(-A) \quad \text{for all } A \in \R^\Rc.
\end{equation}

The mirror symmetry \cref{eq:mirrorsymV} is already implicit in our general
assumptions for the square lattice (as it can be decomposed into a point
reflection and an in-plane rotation by $\pi$). But it is an additional
assumption for the triangular lattice. Here, it is equivalent to strengthen the
rotational symmetry to rotations by $\pi/3$ instead of just $2\pi/3$.

Since $A$ represents an anti-plane {\em displacement gradient} $Du$,
  the map $A \mapsto -A$ does not represent a change in frame as it would in a full three-dimensional setting. In particular the derivation of $V$ for the BCC case in
  \Cref{BCC} shows that \cref{eq:mirrorsymV} is a non-trivial restriction on $V$.

  Indeed, if one derives $V$ from an underlying three-dimensional site potential  (see \Cref{BCC} for such a derivation in the case of a BCC lattice), then \cref{eq:mirrorsymV} means precisely that the three-dimensional lattice is mirror symmetric with respect to the plane orthogonal to the dislocation line. This is quite restrictive and effectively only true if the underlying three-dimensional lattice is given as $\Lambda'= \Lambda \times \Z \subset \R^3$ which is a hexagonal or a cubic lattice for $\Lambda = A_{\rm tri} \Z^2$ or $\Lambda = \Z^2$, respectively.

In the next section \Cref{BCC}, we will then consider a situation where \cref{eq:mirrorsymV} fails, by discussing a 111 screw dislocation in a BCC lattice.

Recall from \cref{screwpredictorformula} that the far-field predictor is
given by $\hat{u}(x) = \frac{b}{2\pi}\arg(x-\hat{x})$. Since we now assume that $\hat{x}$
is at the centre of a square or triangle, $\hat{u}$ satisfies
\begin{align} 
	\label{symofpred}
 \hat{u}(L_\La x) &=
 \begin{cases}
	 \hat{u}(x) + \frac{b}{3} \quad (\modulo b), & \text{triangular lattice}, \\
	 \hat{u}(x) + \frac{b}{4} \quad (\modulo b), & \text{square lattice}. \\
 \end{cases}
\end{align}

Motivated by this observation, we  specify an analogous symmetry
{\em assumption} on a general displacement.

\begin{definition}[Inheritance of symmetries]
We say that a displacement $u$ inherits the rotational symmetry of $\hat{u}$ if
\begin{equation} \label{eq:rotsymu}
  u(L_Q x) = u(x) \quad \text{for all } x \in \La.
\end{equation}
\end{definition}

\begin{remark} \label{rem:inheritancerotsym}
	Inheritance of rotational (or other) symmetries would typically follow from
	the corresponding symmetries of $\hat{u}, \La, V$ and uniqueness of an energy
	minimiser (up to a global translation and lattice slips). However,
	due to the severe non-convexity of the energy landscape it cannot be expected
	in general. As an example, note that the line reflection symmetry in the BCC case, discussed in \Cref{BCC}, is not necessarily inherited as is shown in \cite{Wang2014}.
\end{remark}

We can now state the main results of this section. It is particularly noteworthy that they depend on the lattice under consideration. On a square lattice the symmetry only gives one additional order of decay compared to the decay rates in \cite{EOS2016}, while on a triangular lattice we do indeed show that there are two additional orders of decay as observed numerically in \cite[Remark 3.7]{HO2014}. We will confirm this discrepancy in numerical tests in \Cref{numerics}.

\begin{theorem}[Decay with Mirror Symmetry]\label{thm:highsym}
	Let $\La \in \{\Z^2, A_{\rm tri} \Z^2 \}$ and suppose $\La, \hat{x}, \Rc, V$ satisfy all the assumptions from \Cref{model}. Furthermore, assume $V$ satisfies the mirror symmetry \cref{eq:mirrorsymV}. If $\bar{u}$ is a critical point of $\mathcal{E}$ which inherits the rotational symmetry of $\hat{u}$, then we have for $j=1,2$ and all $\lvert x \rvert$ large enough
\begin{equation} \label{eq:symdecaycorrquad}
	\lvert D^j \bar{u}(x) \rvert \lesssim \lvert x \rvert^{-2-j} \log \lvert x \rvert,
	\end{equation}
	if $\La = \Z^2$, and
	\begin{equation} \label{eq:symdecaycorrtri}
	\lvert D^j \bar{u}(x) \rvert \lesssim \lvert x \rvert^{-3-j}
	\end{equation}
	for the triangular lattice $\La =  A_{\rm tri} \Z^2$.
\end{theorem}

\begin{remark}
  The result is also expected to hold for $j\geq 3$ and $j=0$ (up to subtracting a constant) following ideas in \cite{EOS2016}. As we want to focus on other aspects and do not want to overburden the proof, this is omitted here.
\end{remark}

\begin{remark}
  In the case of a triangular lattice the existence of a critical $u$ has been
  proven in \cite{HO2014} under severe restrictions on $V$. Under further
  restrictions it is even known to be a stable global minimiser. However it is
  unclear whether it inherits symmetry.

  We emphasize, however, that the distinction between the hexagonal and BCC
  lattices, that is the loss of mirror symmetry in the BCC lattice was missed in
  \cite{HO2014}. Therefore, the results of \cite{HO2014} do not apply
  to the BCC case without further work.
\end{remark}

\begin{proof}[Idea of the proof of \Cref{thm:highsym}]
The full proof can be found in \Cref{Proofs}; here we only give a brief idea of the strategy.

Far from the defect core the equilibrium configuration is close to a homogeneous lattice, hence, the linearised problem becomes a good approximation.
  Therefore, a natural quantity to consider is the \emph{linear residual}
\begin{equation} \label{eq:linearresidual}
  f_{u} = - \divRc(\nabla^2V(0)[\DRc u]).
\end{equation}
On the one hand, one can recover $\bar{u}$ as a lattice convolution $\bar{u}=G
\ast_\Lambda f_{\bar{u}}$ where $G$ is the fundamental solution, or Green's function, of
the linear atomistic equations. On the other hand, the decay of $f_{\bar{u}}$
can be estimated by Taylor expansion with the help of the nonlinear atomistic
equations for $\hat{u} + \bar{u}$ and the continuum linear system for $\hat{u}$.
In this expansion, $\nabla^3V(0)=0$ vanishes due to anti-plane symmetry, while
the rotational symmetry leads to simple generic forms of higher order terms.

But even if $f_{\bar{u}}$ decays rapidly, this does not automatically translate to decay
for $\bar{u}=G \ast f_{\bar{u}}$. Even if $f_{\bar{u}}$ has compact support $\bar{u}$ typically only inherits the decay of $G$. However, we show that, due to rotational symmetry, the first moment of $f_{\bar{u}}$ vanishes, while the second has a very special form. Improved estimates for the decay of $f_{\bar{u}}$ together with vanishing moments then lead to an improved rate of decay of $\bar{u}$.

The difference between the triangular lattice and the quadratic lattice lies in
the form of the higher order terms in the expansion of $f_{\bar{u}}$. The terms in
question are given by the atomistic-continuum error of the linear equation and
by the nonlinearity $\nabla^4V(0)$. For the triangular lattice one finds the
leading order expression $c_1 \Delta^2 \hat{u}$ for the linear and  $c_2 (g(x)
\Delta \hat{u} + H(\hat{u}))$ for the nonlinear part, where only the constants
$c_1, c_2$ depend on the potentials. Here $H$ is the mean curvature of the graph
$(x_1, x_2, \hat{u}(x))^T$. And the mean curvature vanishes as the graph is a helicoid, a minimal surface. Since $\Delta^2 \hat{u}=0$, $\Delta \hat{u}=0$, and
$H(\hat{u})=0$ all the leading order terms vanish. On the other
hand, for the quadratic lattice, these terms are nontrivial and do not cancel.
\end{proof}


\subsection{Anti-plane screw dislocation in BCC} \label{BCC}
We turn towards the physically more important setting of a straight screw
dislocation along the 111 direction in a BCC crystal. The three-dimensional BCC lattice can be defined by $\La'' = \Z^3 + \{0,p\}$, with shift $p=  \begin{pmatrix}
\frac{1}{2} & \frac{1}{2} & \frac{1}{2} \\
\end{pmatrix}^T$.
A screw dislocation along the 111 direction is obtained by taking both
dislocation line and Burgers vector parallel to the vector $(1,1,1)^T$.
If we rotate $\La''$ by
\[Q = \frac{1}{\sqrt{6}}\begin{pmatrix}
-1 & -1 & 2\\ \sqrt{3} & -\sqrt{3} & 0\\ \sqrt{2} & \sqrt{2} & \sqrt{2}
\end{pmatrix} \]
and then rescale the lattice by $\sqrt{3/2}$, we obtain the
three-dimensional Bravais lattice
\[ \La' = \sqrt{3/2} Q \La'' = \begin{pmatrix}
1 & \frac{1}{2} & 0 \\ 0 & \frac{\sqrt{3}}{2} & 0 \\ \frac{1}{2 \sqrt{2}} & -\frac{1}{2 \sqrt{2}} & \frac{3}{2\sqrt{2}}
\end{pmatrix} \Z^3.\]
The 111 direction becomes the $e_3$ direction under this transformation, which is convenient for the subsequent discussion.

Since $p = \frac{3}{2\sqrt{2}}$, the {\em Burgers vector} is now given by $b = \pm \frac{3}{2\sqrt{2}}$ (corresponding to the actual {\em Burgers vector} in three dimensions being $(0,0,b)^T$). We project the BCC lattice $\La'$ along the dislocation direction $e_3$ to obtain the triangular lattice
\[
	\La = \big\{ (x_1,x_2) \,\big|\, x \in \La' \big\} = A_{\rm tri} \Z^2.
\]

Note though that these projections correspond to different ``heights'', i.e.,
different $z$-coordinates in $\La'$. Indeed, it is helpful to split $\La$ into the three lattices $\La = \La_1 \cup \La_2 \cup \La_3$, where
\[\La_i = v_i + \begin{pmatrix} \frac{3}{2} & \frac{3}{2} \\ \frac{\sqrt{3}}{2} &-\frac{\sqrt{3}}{2}
\end{pmatrix} \Z^2,\]
with $v_1 =0$, $v_2 = e_1$, $v_3 = \begin{pmatrix}
\frac{1}{2} & \frac{\sqrt{3}}{2}
\end{pmatrix}^T$. In this notation, one can recover the three-dimensional lattice as
\[\La' = \bigcup_i \Big(\La_i \times \Big\{ \Big(k + \frac{i}{3}\Big) \frac{3}{2\sqrt{2}} \colon k \in \Z \Big\} \Big).  \]

\begin{figure}[!htbp]
\centering
\begin{minipage}{0.32\textwidth}
\centering
\includegraphics[width=0.8\linewidth]{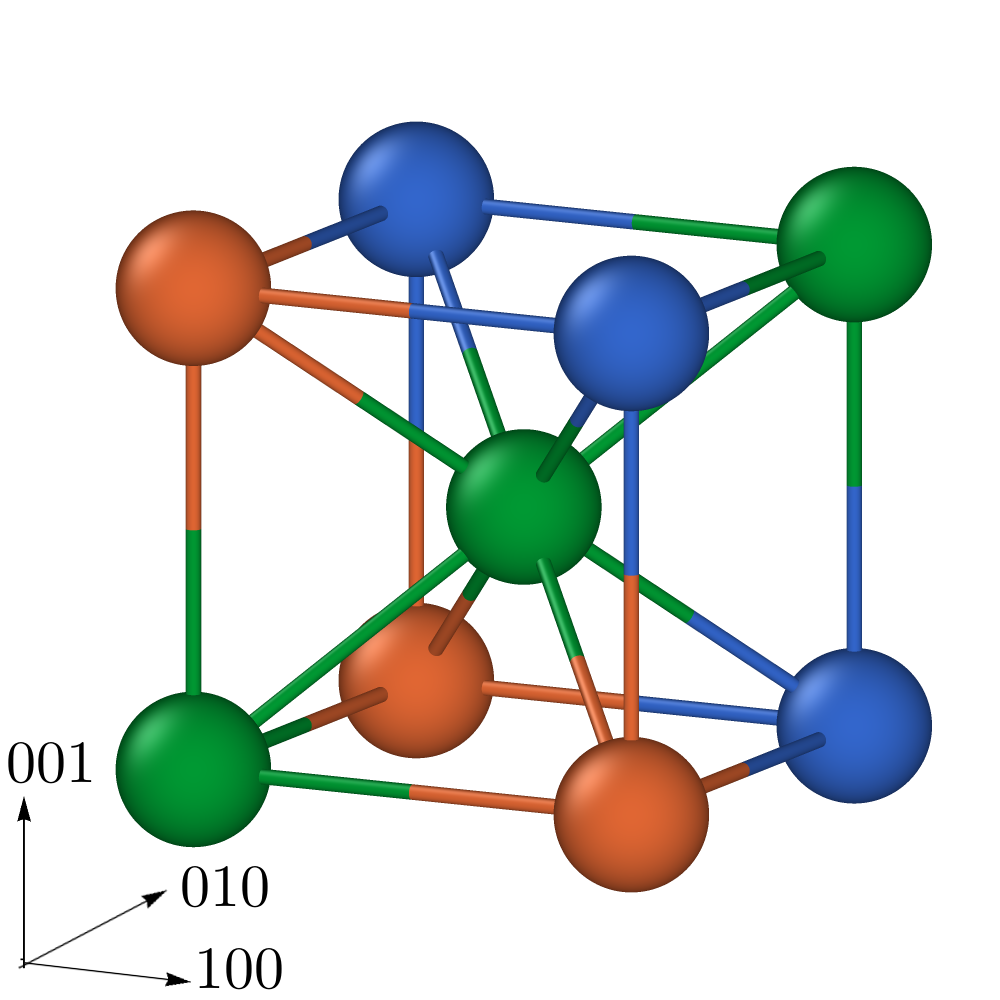}
\label{fig:bbc_unit}
\end{minipage}
\begin{minipage}{0.32\textwidth}
 \centering
\includegraphics[width=0.8\linewidth]{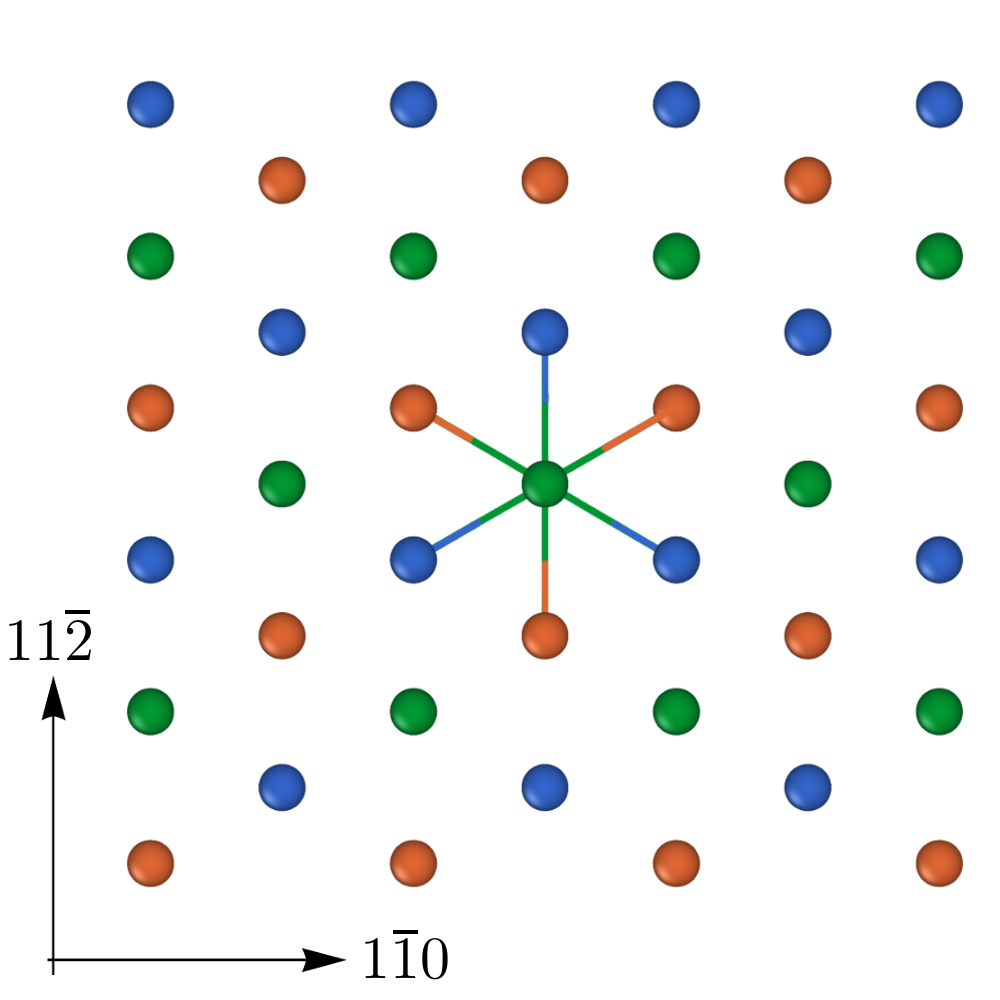}
\label{fig:bcc_111}
\end{minipage}
\begin{minipage}{0.32\textwidth}
\includegraphics[width=0.86\linewidth]{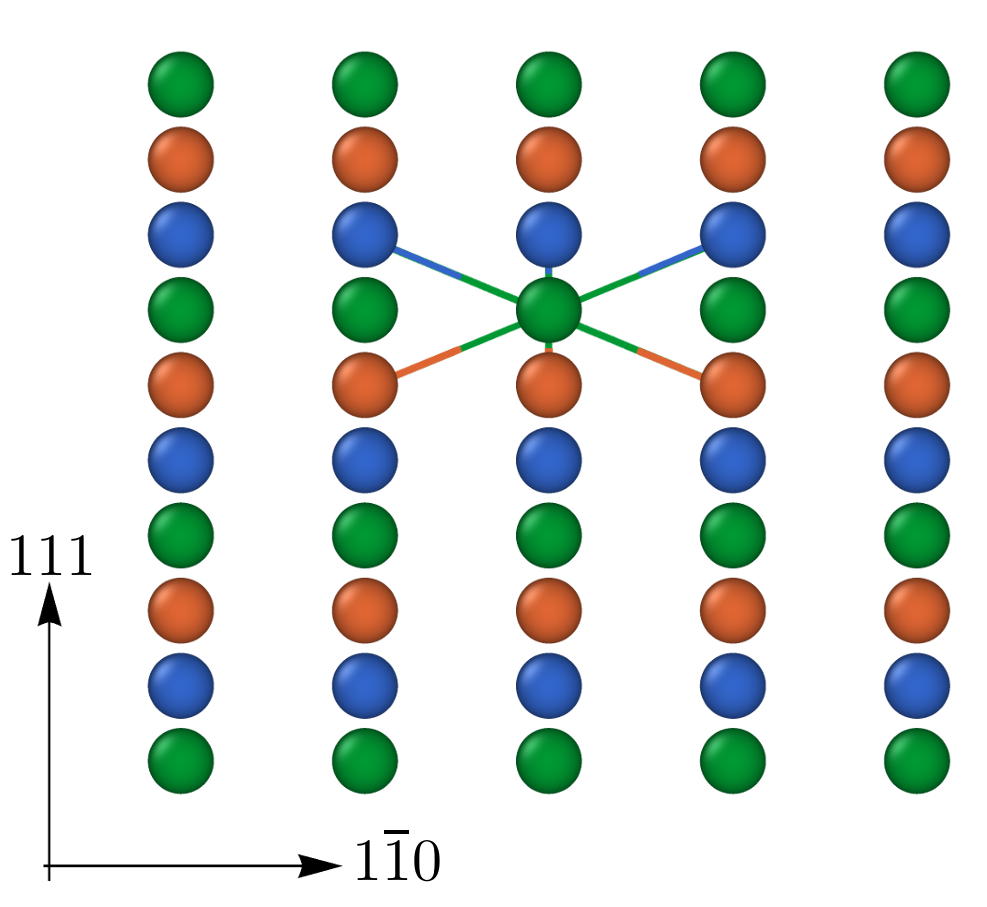}
\label{fig:bcc_11-2}
\end{minipage}
\caption[caption]{Consider the middle green atom in the BCC unit cube (left picture). After projecting along the 111-direction (the green diagonal), the three green atoms are represented as one, which has six other-coloured atoms in the unit cube as its nearest-neighbours (middle picture). The different heights of atomic planes associated with each colour are best seen by projecting the same lattice along the $11\overline{2}$-direction (right picture).}
\end{figure}

Next, we formally derive an anti-plane interatomic potential as a projection
from a three-dimensional model. The derivation is only formal as many of the sums appearing are infinite if summed over the entire lattice. Indeed, for a deformation $y$ consider formally
\[\mathcal{E}^{\rm 3d}(y) = \sum_{x \in \La'} V'(D' y(x)),\]
where $D'y(x) = (D_\rho y(x))_{\rho \in \Lambda'}$ and $V' \colon \R^{3 \times
\Lambda'} \to \R$. Note that, to achieve the periodicity of $V$ (slip
invariance) $V'$ must depend on the entire crystal. However, it is
convenient to assume that it has a finite cut-off $d>0$ such that $V'(A) =
V'(B)$ whenever $A,B$ satisfy $A_\rho=B_\rho$ for all $\rho$ with $\lvert A_\rho \rvert<d$ or $\lvert B_\rho \rvert < d$.

In contrast to $\mathcal{E}$, $\mathcal{E}^{\rm 3d}$ acts on deformations instead of displacements. To derive an energy on anti-plane displacements,
we consider deformations of the form
\[
	y^u : \La' \to \R^3, \qquad
	y^u(x) := \big(x_1, x_2, x_3 + u(x_1, x_2)\big)
\]
for anti-plane displacements $u : \La \to \R$. As differences of $y^u$ do not depend on $x_3$, the same is true for the local energy contributions. Therefore, we can formally renormalise the (possibly infinite) energy to
\[\mathcal{E}^{\rm 3d}_{\rm norm}(u) = \sum_{x \in \La'\cap (\R^2\times [0,p))}  V'(D' y^u(x)),\]
the energy per periodic layer of thickness $p$. Since $\lvert D_\rho y^u(x)\rvert \geq \lvert (\rho_1, \rho_2) \rvert$, the local energy at any $x$ can only depend on the projected directions $\mathcal{R} :=\La \cap B_d(0) \backslash\{0\}$. We can therefore define
\[
	V(Du(x_1, x_2)) := V'(D' y^u(x)),
\]
for $x \in \La'$, to obtain $\mathcal{E}(u) = \mathcal{E}^{\rm 3d}_{\rm norm}(y^u)$.

Of course we assume that $V'$ is frame-indifferent, $V'(QA) = V'(A)$ for all $A$
and $Q \in \Oo(3)$. Furthermore, we assume that $V'$ is invariant under relabelling of atoms (permutation invariance). In particular, this means that $V'$ is compatible with the lattice symmetries of $\La'$: $V'(A)=V'((A_{-\rho})_{\rho \in \Lambda'})$ and $V'(A)=V'((A_{Q'\rho})_{\rho \in \Lambda'})$, where $Q'$ is the rotation through $2 \pi/3$ with axis $e_3$. $\La'$ is also invariant under line reflection symmetry with respect to the line spanned by $a'= (\frac{\sqrt{3}}{2}, \frac{1}{2}, 0)^T$. Denoting the reflection map by $S'$ we thus have $V'(A)=V'((A_{S'\rho})_{\rho \in \Lambda'})$.

We can now translate these properties to symmetries of $V$. Clearly, $\mathcal{R}=-\mathcal{R}$ and $Q_\La \Rc =\Rc$. The symmetry properties of $V'$ directly imply $V(A) = V\left((A_{Q_\La \rho})_{\rho \in \Rc}\right)$ and $V(A) =
V\left((-A_{-\rho})_{\rho \in \Rc}\right)$ for all $A \in \R^\Rc$. The slip
invariance $V(A) = V(A + p(\delta_{\rho \sigma})_{\sigma \in \Rc})$ also follows from permutation invariance of $V'$. We have thus obtained all the general
assumptions that we imposed on $V$ in \Cref{model}.

Additionally, we will exploit the line reflection symmetry. Let $a= (\frac{\sqrt{3}}{2}, \frac{1}{2} )^T$. A reflection at the line spanned by $a$ in $\R^2$ is given by
\[S = a \otimes a - a^\bot \otimes a^\bot=\begin{pmatrix}
\frac{1}{2} &\frac{\sqrt{3}}{2} \\ \frac{\sqrt{3}}{2} & - \frac{1}{2}
\end{pmatrix}. \]
Due to the line reflection symmetry described by $S'$ as well as frame-indifference with $Q=S'$, we deduce
\begin{equation} \label{eq:linereflectionsym}
S \Rc =\Rc, \quad \text{and} \quad  V(A) = V\left((-A_{S \rho})_{\rho \in \Rc}\right).
\end{equation}

We emphasize that $\La'$ is {\em not invariant} under a rotation by only $\pi/3$ around the axis $e_3$. This is easily seen, as this rotation maps $\Lambda_2$ to $\Lambda_3$ and vice versa. Equivalently, it is not invariant under the mirror symmetry $x \mapsto (x_1, x_2, - x_3)^T$ expressed by \cref{eq:mirrorsymV}. Therefore, the more specific results from the previous section, \Cref{symsec}, do not apply.

While in the setting of \Cref{symsec} screw dislocations with Burgers vector $b=p$ and $b=-p$ are equivalent, the loss of mirror symmetry in the BCC crystal also creates two distinctively different screw dislocations, the so-called easy and hard core. In particular, they have a different core structure; see e.g. \cite{bcc-easy-hard}.

The improved decay rates we obtained in \Cref{symsec} no longer hold up either.
Indeed, one can see in numerical calculations, see \Cref{numerics}, that the
$\lvert x \rvert^{-2}$ bound on the decay of the strains is sharp (up to
logarithmic terms and constants).

Our aim now, as announced in the Introduction, is to develop a new far-field
predictor so that the corresponding corrector recovers the higher $\lvert x
\rvert^{-4}$ accuracy of the more symmetric case. A natural first idea is to
replace CLE with the Cauchy--Born nonlinear elasticity equation, however,
these are not easy to solve analytically. Instead, we expand the solution
$u = \hat{u} + u_1 + u_2 + \dots$ hoping for $\nabla^j u_2 \ll \nabla^j u_1 \ll \nabla^j \hat{u}$, which yields
\begin{align*}
  \divo \nabla W(\nabla u) \sim\,&
  \divo \nabla^2 W(0) \nabla \hat{u} \\
  & \hspace{-2cm}+ \divo \Big( \nabla^2 W(0) \nabla u_1 + \frac{1}{2}\ \nabla^3 W(0)[\nabla \hat{u}, \nabla \hat{u}] \Big) \\
  & \hspace{-2cm} + \divo \Big( \nabla^2 W(0) \nabla u_2 + \nabla^3 W(0)[\nabla \hat{u}, \nabla u_{1}] + \nabla^4 W(0) [\nabla \hat{u}, \nabla \hat{u}, \nabla \hat{u}] \Big)
    + \dots .
\end{align*}

The atomistic-continuum error is typically expected to be of comparable size as the last terms. But, as the projected lattice is still a triangular lattice, many of the arguments discussed in \Cref{symsec} still apply and the highest order of this error as well as the term $\nabla^4 W(0) [\nabla \hat{u}, \nabla \hat{u}, \nabla \hat{u}]$ vanish. However, we now have $\nabla^3 W(0) \neq 0$ making the remaining terms non-trivial. We can thus obtain the first two corrections to $\hat{u}$ by solving the linear PDEs
\begin{subequations} \begin{align}
-\divo \nabla^2 W(0) \nabla u_1 &= \frac{1}{2} \divo \Big( \nabla^3 W(0)[\nabla \hat{u}, \nabla \hat{u}]\Big),\label{eqcorr1a}\\
-\divo \nabla^2 W(0) \nabla u_2 &= \divo \Big( \nabla^3 W(0)[\nabla \hat{u}, \nabla u_{1}]\Big)\label{eqcorr2a}
\end{align} \end{subequations}
on $\R^2 \backslash \{0\}$.

Due to \Cref{D2W-cor,D3W-cor}, exploiting the rotational crystalline symmetry,
we can simplify them as
\begin{subequations} \begin{align}
- c_{\rm lin} \Delta u_1 &= c_{\rm quad} \begin{pmatrix}
\partial_{11} \hat{u}- \partial_{22} \hat{u} \\ -2 \partial_{12} \hat{u}
\end{pmatrix} \cdot \nabla \hat{u},\label{eqcorr1b}\\
- c_{\rm lin} \Delta u_2 &= c_{\rm quad} \bigg(\begin{pmatrix}
\partial_{11} u_1 - \partial_{22} u_1 \\ -2 \partial_{12} u_1
\end{pmatrix} \cdot \nabla \hat{u}  +  \begin{pmatrix}
\partial_{11} \hat{u} - \partial_{22} \hat{u} \\ -2 \partial_{12} \hat{u}
\end{pmatrix}\cdot\nabla u_1 \bigg). \label{eqcorr2b}
\end{align} \end{subequations}
where
\begin{align*}
c_{\rm lin} &=\frac{1}{2} \tr \nabla^2W(0), \qquad \text{and} \\
c_{\rm quad} &=\frac{1}{4} (\nabla^3W(0)_{111} - 3\nabla^3W(0)_{122}). \\
\end{align*}

In polar coordinates, $x = \hat{x} + r (\cos\varphi, \sin\varphi)$, using the
fact that $\hat{u}=\frac{b}{2 \pi} \arg(x-\hat{x}) = \frac{b}{2 \pi} \varphi$,
\cref{eqcorr1b} becomes
\[
  - \Delta u_1 = \frac{c_{\rm quad} b^2}{c_{\rm lin} 2 \pi^2} \frac{\cos(3\varphi)}{r^3},
\]
from which we readily infer that one possible solution is
\begin{equation}\label{u1form}
  u_1(x+\hat{x}) = \frac{c_{\rm quad} b^2}{c_{\rm lin} 16 \pi^2}\frac{\cos(3\varphi)}{r} = \frac{c_{\rm quad} b^2}{c_{\rm lin} 16 \pi^2} \frac{x_1^3 - 3 x_1 x_2^2}{\lvert x \rvert^4}.
\end{equation}


Similarly, inserting $\hat{u}$ and $u_1$ into \cref{eqcorr2b} yields
\[
  - \Delta u_2 = \frac{c_{\rm quad}^2 b^3}{c_{\rm lin}^2 4 \pi^3}
  \frac{\sin(6 \varphi)}{r^4},
\]
for which a solution is given by
\begin{equation}
u_2(x+ \hat{x}) =  \frac{c_{\rm quad}^2 b^3}{c_{\rm lin}^2 128 \pi^3} \frac{\sin(6\varphi)}{r^2} =  \frac{c_{\rm quad}^2 b^3}{c_{\rm lin}^2 128 \pi^3} \frac{6x_1^5 x_2 - 20 x_1^3 x_2^3 +6 x_1 x_2^5}{\lvert x\rvert^8}\label{u2form}.
\end{equation}

While there are many more solutions for both problems, we will choose these specific ones as they satisfy the decay estimates
\begin{equation} \label{eq:u1u2estimates}
\lvert \nabla^j u_i \rvert \lesssim \lvert x \rvert^{-i-j}
\end{equation}
and the rotational symmetry $u_i(L_Q x) = u_i(x)$. With the solutions $u_1$ and $u_2$ obtained, respectively, in
\cref{u1form,u2form} we obtain the following result.

\begin{theorem}[BCC]\label{BCC-result}
Let $\La = A_{\rm tri} \Z^2$ and suppose $\La, \hat{x}, \Rc, V$ satisfy all the assumptions from \Cref{model}. Furthermore, assume $\Rc$ and $V$ satisfy the line reflection symmetry \cref{eq:linereflectionsym}. Consider a critical point $\bar{u}$ of \cref{energy} that inherits the rotational symmetry of $\hat{u}$. Then we can write $\bar{u} = u_1 + u_2 + \bar{u}_{\rm rem}$ where $u_1$ and $u_2$ are given by \cref{u1form} and \cref{u2form} and the remainder $\bar{u}_{\rm rem}$ satisfies the decay estimates
\begin{equation}\label{BCC-est}
    \lvert D^j \bar{u}_{\rm rem}(x) \rvert \lesssim \lvert x \rvert^{-j-3}\log \lvert x \rvert,
\end{equation}
for $j=1,2$ and all $\lvert x \rvert$ large enough.
\end{theorem}
\begin{remark}
As discussed in the introduction, our new predictor $\hat{u} + u_1 + u_2$ does not just result in $O(\lvert x \rvert^{-3})$ accuracy for the strain which one might expect from the general expansion idea or from well-established results about the Cauchy-Born anti-discretisation error. The actual accuracy is one order higher, i.e., $O(\lvert x \rvert^{-4})$.
\end{remark}
\begin{remark}
  Since $|D^j u_1(x)| \lesssim |x|^{-j-1}$, {\em without} log-factors,
  \Cref{BCC-result} improves the result of \Cref{EOSdecay} to
  \[
      \big|D^j \bar{u}(x)\big|
        \lesssim |x|^{-j-1}, \qquad j = 1, 2.
  \]
 \end{remark}

\section{Numerical approximation} \label{numerics}
\subsection{Supercell approximation}
\label{sec:approx}
A central motivation for the present work are the poor convergence rates of
standard supercell approximations for the defect equilibration problem
\cref{eq:minproblem} established in \cite{EOS2016}. We can now exploit the
theoretical results from \Cref{sec:mainresults} to construct boundary conditions
that give rise to new supercell approximations. These have improved rates of
convergence without any corresponding increase in computational complexity.

\def\ui{\tilde{y}}
\def\us{\tilde{u}}
\def\L{\Lambda}

We begin by defining a generalised energy-difference functional in a
predictor-corrector form
\begin{align*}
	\mathcal{E}(u_{\rm pred}; u) :=
			\sum_{x \in \L} V\big(Du_{\rm pred}(x) + Du(x)\big)
									- V\big(Du_{\rm pred}(x)\big), &\\
   \text{for } u_{\rm pred} \in \hat{u} + \Hcc,  u \in \Hcc.&
\end{align*}
Then, the generalised variational problem
\begin{equation} \label{eq:full-problem}
	\us \in \arg\min \big\{ \mathcal{E}(u_{\rm pred}; u) \, |\, u \in \Hcc \big\}
\end{equation}
is equivalent to \cref{eq:minproblem}, via the identity
$u_{\rm pred} + \us = \hat{u} + \bar{u}$.

We now note as in \cite{EOS2016} that the supercell approximation on a
domain $B_R \cap \L \subset \Omega_R \subset \L$ with boundary condition $u_{\rm pred}$ on $\L \setminus \Omega_R$ can be written as a Galerkin approximation
\begin{align}
	\label{eq:approx-prob}
	&\us_R \in \arg\min \big\{ \mathcal{E}(u_{\rm pred}; u)\,|\,
						u \in \mathcal{H}^0(\Omega_R) \big\},   \\
   \notag
   \text{where} \qquad
	&\mathcal{H}^0(\Omega_R)
	:= \{ v \in \mathcal{H}^c\,|\,v = 0\,\text{ in }\, \La \setminus \Omega_R\}.
\end{align}

Using generic properties of Galerkin approximations we obtain the
following approximation error estimate.

\begin{theorem}\label{approx-error-thm}
	Let $\us$ be a \textit{strongly stable solution} (cf. \cite{EOS2016}) to \cref{eq:full-problem}, i.e. satisfying
	\[
	 \delta_u^2 \mathcal{E}(u_{\rm pred};\us)[v,v] \geq C \|v\|_{\Hcc}^2,
	\]
for all $v \in \mathcal{H}^c$. If $\us$ further satisfies
	 \[
	   |D\us(x)| \lesssim |x|^{-s}\log^r|x|,
	 \]
	for some $s > 1, r \in \{0, 1\}$, then there exist $C, R_0 > 0$ such that,
	for all $R > R_0$ there exists a stable solution $\us_R$ to
	\cref{eq:approx-prob} satisfying
	 \begin{equation}
	   \|\us_R - \us\|_{\Hcc} \leq C R^{-s+1} \log^r(R).
	 \end{equation}
\end{theorem}
\begin{proof}
	The existence of a solution $\us_R$, for $R$ sufficiently large, follows
	from \cite[Theorem 6]{EOS2016} (the case $u_{\rm pred}  = \hat{u}$) and the
	equivalence of \cref{eq:full-problem} with \cref{eq:minproblem}.
	Moreover, following the proof of \cite[Theorem 6]{EOS2016} verbatim we obtain
	\begin{align*}
		\|\us_R - \us\|_{\Hcc} \lesssim \|\us\|_{\Hcc(\Lambda\setminus B_{R/2})}.
	\end{align*}
	We then apply the assumption that $|D\us(x)| \lesssim |x|^{-s}\log^r|x|$ to
	arrive at the desired error estimate,
	\begin{equation*}
		\|\us_R - \us\|_{\Hcc}
		\lesssim \bigg(\sum_{x\in\La\setminus B_{R/2}} |D\us(x)|^2\bigg)^{1/2}
		\lesssim R^{1-s} \log^r(R).
	\end{equation*}
\end{proof}

\subsection{Numerical examples with mirror symmetry}\label{numerics-simple}
To test the results from \Cref{symsec} we consider a toy model involving nearest-neighbour pair interaction,
\[
 	V(Du(x)) = \sum_{\rho\in\Rc}\psi(D_{\rho}u(x)),
	\qquad \psi(r) = \sin^2(\pi r),
\]
which is $1$-periodic, i.e., $p = 1$. We investigate the three cases
\begin{enumerate}[label=(\roman*)]
 \item symmetric square: \\
 		$\La = \Z^2,\quad\Rc = \{\pm e_1, \pm e_2\},\quad \hat{x} = \left(\frac{1}{2},\frac{1}{2}\right)$;
 \item symmetric triangular: \\
 		$\La = A_{\rm tri}\Z^2,\quad\Rc = \left\{\pm   \begin{pmatrix}
  1 \\
  0
 \end{pmatrix}, \pm   \begin{pmatrix}
  \frac{1}{2} \\
  \frac{\sqrt{3}}{2}
 \end{pmatrix},\pm   \begin{pmatrix}
  -\frac{1}{2} \\
  \frac{\sqrt{3}}{2}
 \end{pmatrix}\right\},\quad \hat{x} = \left(\frac{1}{2},\frac{\sqrt{3}}{6}\right)$;
 \item asymmetric triangular: as in (ii), but with $\hat{x} = \left(\frac{1}{4}, \frac{1}{8}\right)$.
\end{enumerate}
The cases (i) and (ii) satisfy all conditions of \Cref{thm:highsym}
while (iii) fails the crucial symmetry assumptions. In particular, at least up to logarithmic terms,  our theory predicts $\lvert D \bar{u}(x) \rvert \lesssim \lvert x \rvert^{-3}$ for (i), $\lvert D \bar{u}(x) \rvert \lesssim \lvert x \rvert^{-4}$ for (ii), and $\lvert D \bar{u}(x) \rvert \lesssim \lvert x \rvert^{-2}$ for (iii). Due to \Cref{approx-error-thm} this corresponds to $\|\us_R - \us\|_{\Hcc}$ being $O(R^{-2})$, $O(R^{-3})$, and  $O(R^{-1})$, respectively.
To compute equilibria we employ a standard Newton scheme, terminated at an
$\ell^{\infty}$-residual of $10^{-8}$. In \Cref{fig:symm_decay} we plot both
the decay of the correctors, confirming the predictions of
\Cref{thm:highsym}, and the approximation error in the supercell approximation against the domain size $R$, confirming the prediction of \Cref{approx-error-thm}.
\begin{remark}
 An asymmetric square case (that is as in (i) but with $\hat{x} = \left(\frac{1}{3},\frac{1}{3}\right)$ has also been considered and the results are as expected by our theory and thus are qualitatively equivalent to (iii). Therefore we do not include them in the figures to retain clarity. It does however further emphasise the role of symmetry in the problem.
\end{remark}

\begin{figure}[!htbp]
\centering
\begin{minipage}{0.49\textwidth}
\centering
\includegraphics[width=0.96\linewidth]{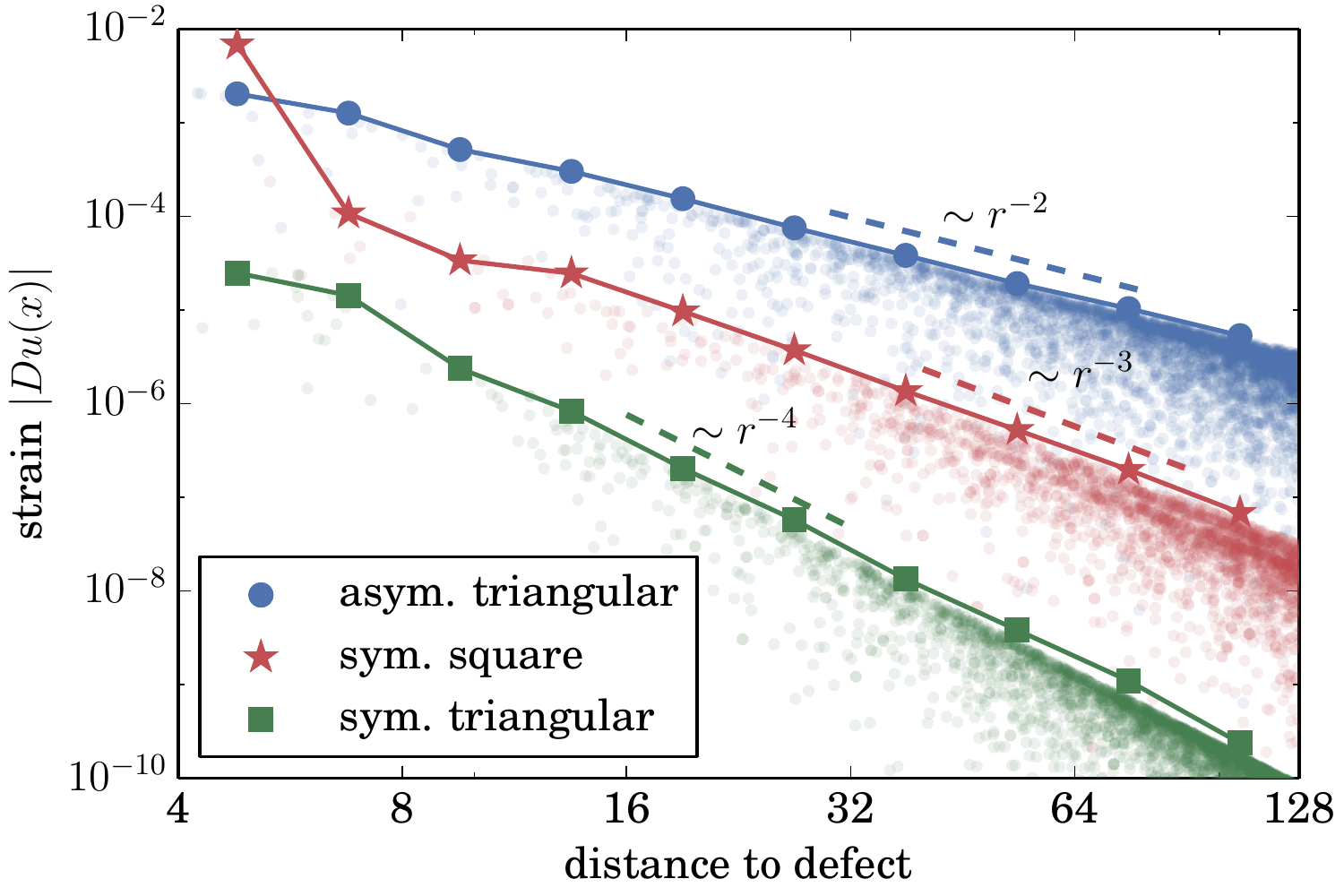}
\label{fig:symm_decay}
\end{minipage}
\begin{minipage}{0.49\textwidth}
 \centering
\includegraphics[width=0.96\linewidth]{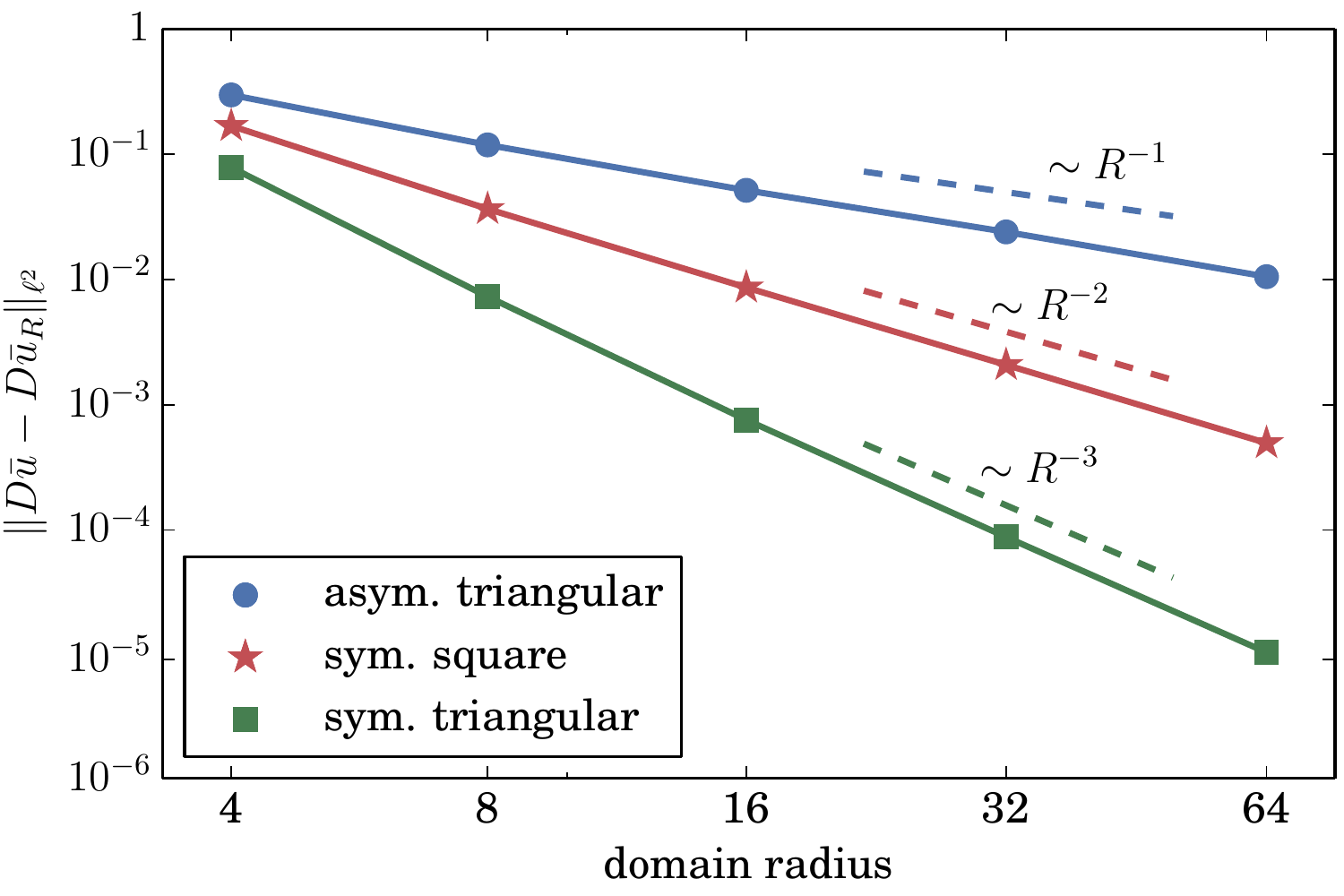}
\label{fig:symm_err}
\end{minipage}
\caption[caption]{Left: Decay of $|\DRc\bar{u}|$ for the triangular and square lattices, with and without rotational symmetry. Transparent dots denote data points $(|x|, |Du(x)|)$, solid curves their envelopes. We observe the improved decay rates $r^{-3}$ and $r^{-4}$, proven in \Cref{thm:highsym}, when the dislocation core is chosen as a high symmetry point.\\\hspace{\textwidth}Right: Rates of convergence of the supercell approximation \cref{eq:approx-prob} in the three cases specified in \Cref{numerics-simple}. We observe the improved rates of decay of the corrector in the high symmetry cases as predicted by \Cref{approx-error-thm}. }
\end{figure}

\subsection{Numerical example in BCC Tungsten}
\label{sec:numerics_bcc}
To confirm the result of \Cref{BCC}, we consider a Finnis--Sinclair type
model (EAM model) for BCC Tungsten (W), where the 3D site energy for a deformation $y$ is of the form
\[
	V'(D'y) = -\Big( \sum_{\sigma \in \Lambda'} \rho\big( |D_\sigma y| \big) \Big)^{1/2}
				+ \sum_{\sigma \in \Lambda'} \phi\big( |D_\sigma y| \big),
\]
and the electron density $\rho$ and pair repulsion $\phi$ are obtained
from \cite{Wang2014}. The projected anti-plane model is then constructed
as described in \Cref{BCC}. The supercell model \cref{eq:approx-prob} is
solved to within an $\ell^\infty$ residual of $10^{-6}$ using a preconditioned
LBFGS algorithm \cite{2016-precon1}.

We investigate two test cases, the easy dislocation core (negatively oriented) and
the hard dislocation core (positively oriented), cf. \cite{bcc-easy-hard}. For each case, following \Cref{BCC}, we consider three different predictors:
\begin{enumerate}[label=(\roman*)]
	 \item standard linearised elasticity predictor (0th order), i.e., $u_{\rm pred} = \hat{u}$;
	 \item 1st order correction, i.e. $u_{\rm pred} = \hat{u} + u_1$;
	 \item 2nd order correction, i.e. $u_{\rm pred} = \hat{u} + u_1 + u_2$,
\end{enumerate}
with $\hat{u}$ given in \cref{screwpredictorformula} and $u_1, u_2$,
respectively, in \cref{u1form} and \cref{u2form}.

In \Cref{fig:bcc-hard,fig:bcc-easy} on the left-hand side we display the decay of the correctors for, respectively, the hard (positive) and easy (negative) dislocation cores, confirming the prediction of \Cref{BCC-result}. On the right-hand side we plot the corresponding approximation errors in the supercell approximation against the domain size $R$, confirming the prediction of \Cref{approx-error-thm}.
\begin{figure}[!htbp]
\centering
\begin{minipage}{0.49\textwidth}
\centering
\includegraphics[width=0.96\linewidth]{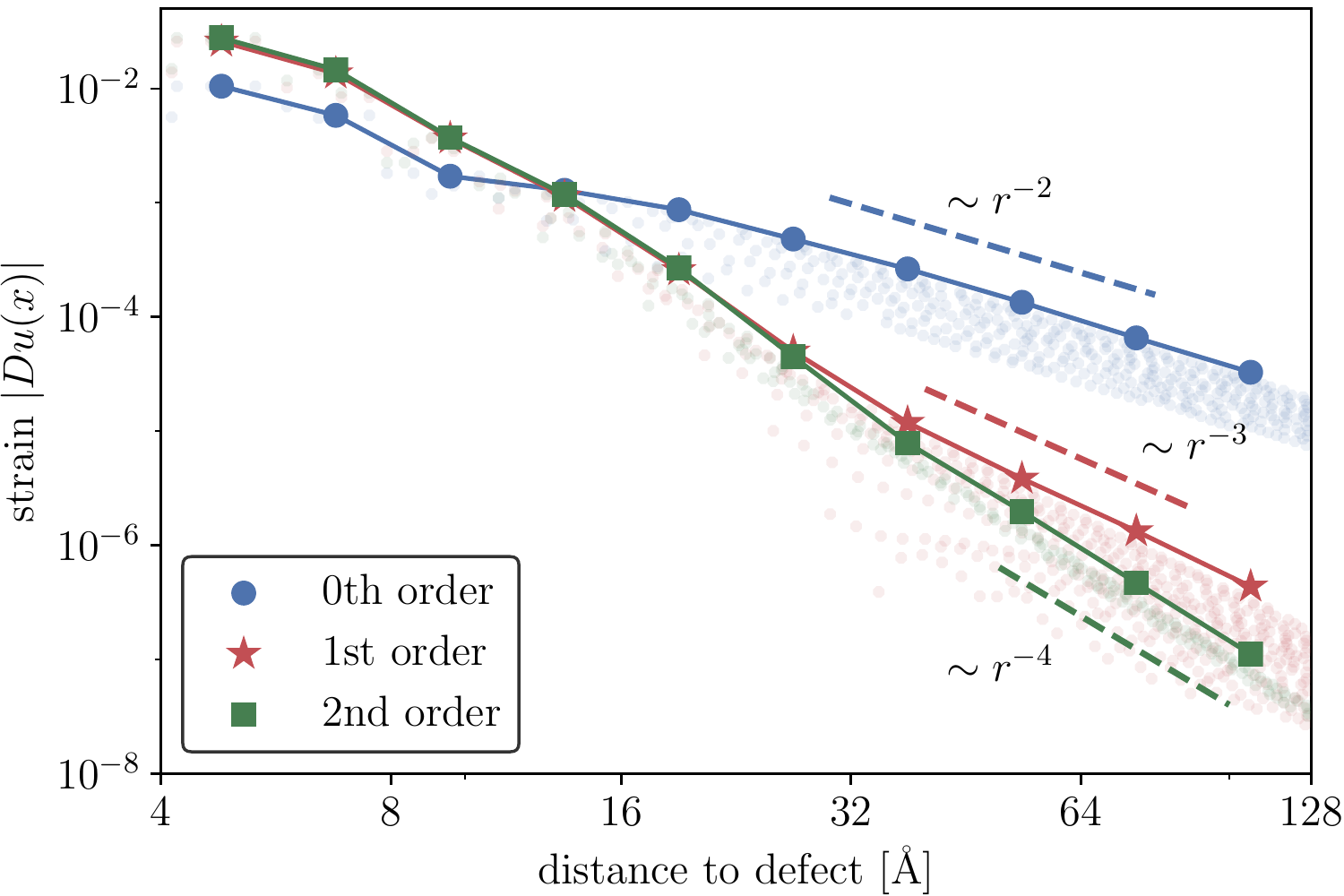}
\end{minipage}
\begin{minipage}{0.49\textwidth}
 \centering
\includegraphics[width=0.96\linewidth]{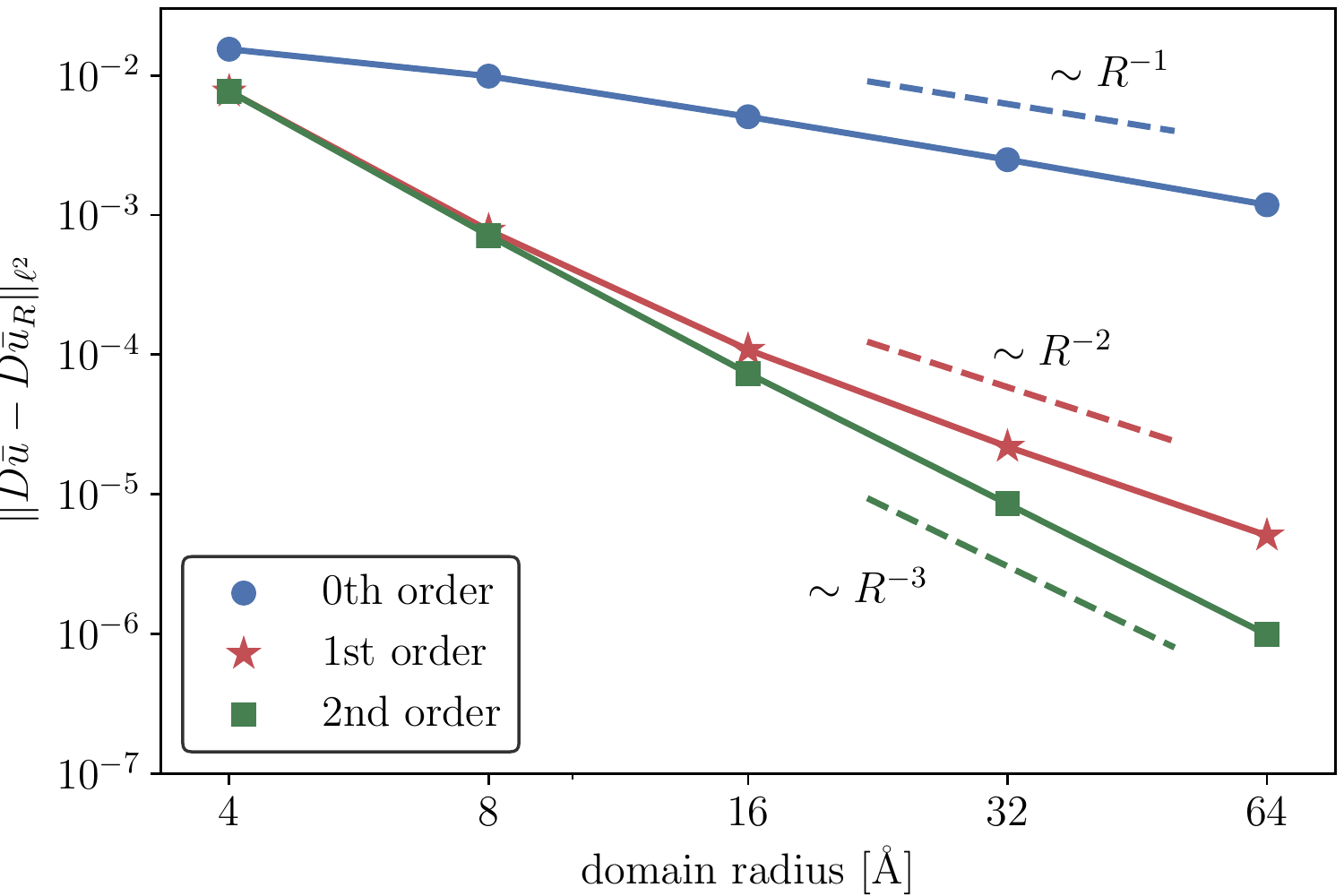}
\end{minipage}
\caption[caption]{Left: Decay of $|\DRc\bar{u}|$ for a BCC easy core screw
dislocation with standard and improved far-field predictors; cf.~\Cref{sec:numerics_bcc}. Transparent dots denote data points $(|x|, |Du(x)|)$, solid curves their envelopes. The numerically observed improved decay for higher-order predictors is consistent with \Cref{BCC-result}.\\\hspace{\textwidth}Right: Rates of convergence of the supercell approximation \cref{eq:approx-prob} to the BCC easy core screw dislocation, employing the standard as well as higher-order far-field predictors. The improved rates of convergence due to the faster decay of the corrector solutions are consistent with \Cref{approx-error-thm}.}
\label{fig:bcc-easy}
\end{figure}
\begin{figure}[!htbp]
\centering
\begin{minipage}{0.49\textwidth}
\centering
\includegraphics[width=0.96\linewidth]{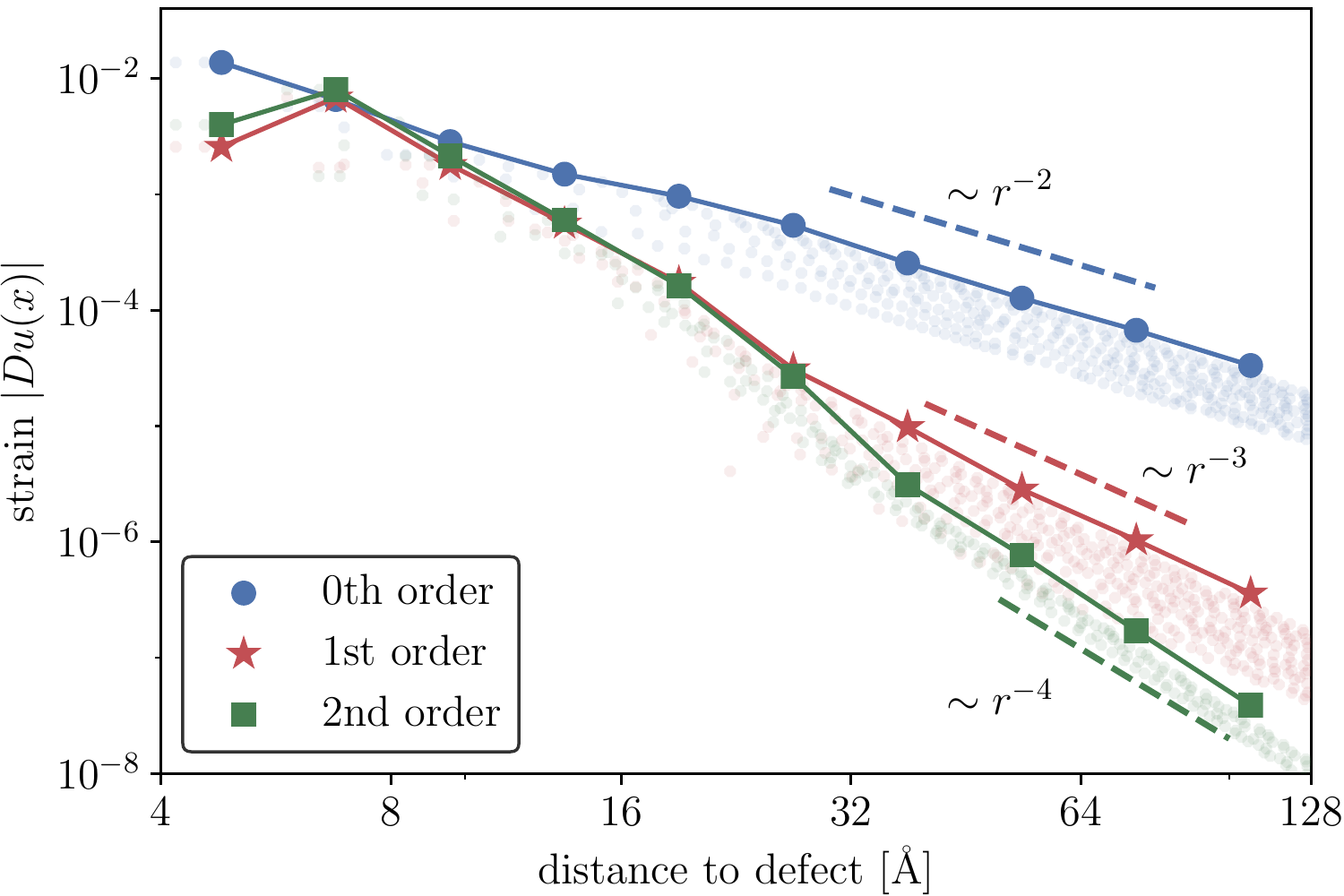}
\end{minipage}
\begin{minipage}{0.49\textwidth}
 \centering
\includegraphics[width=0.96\linewidth]{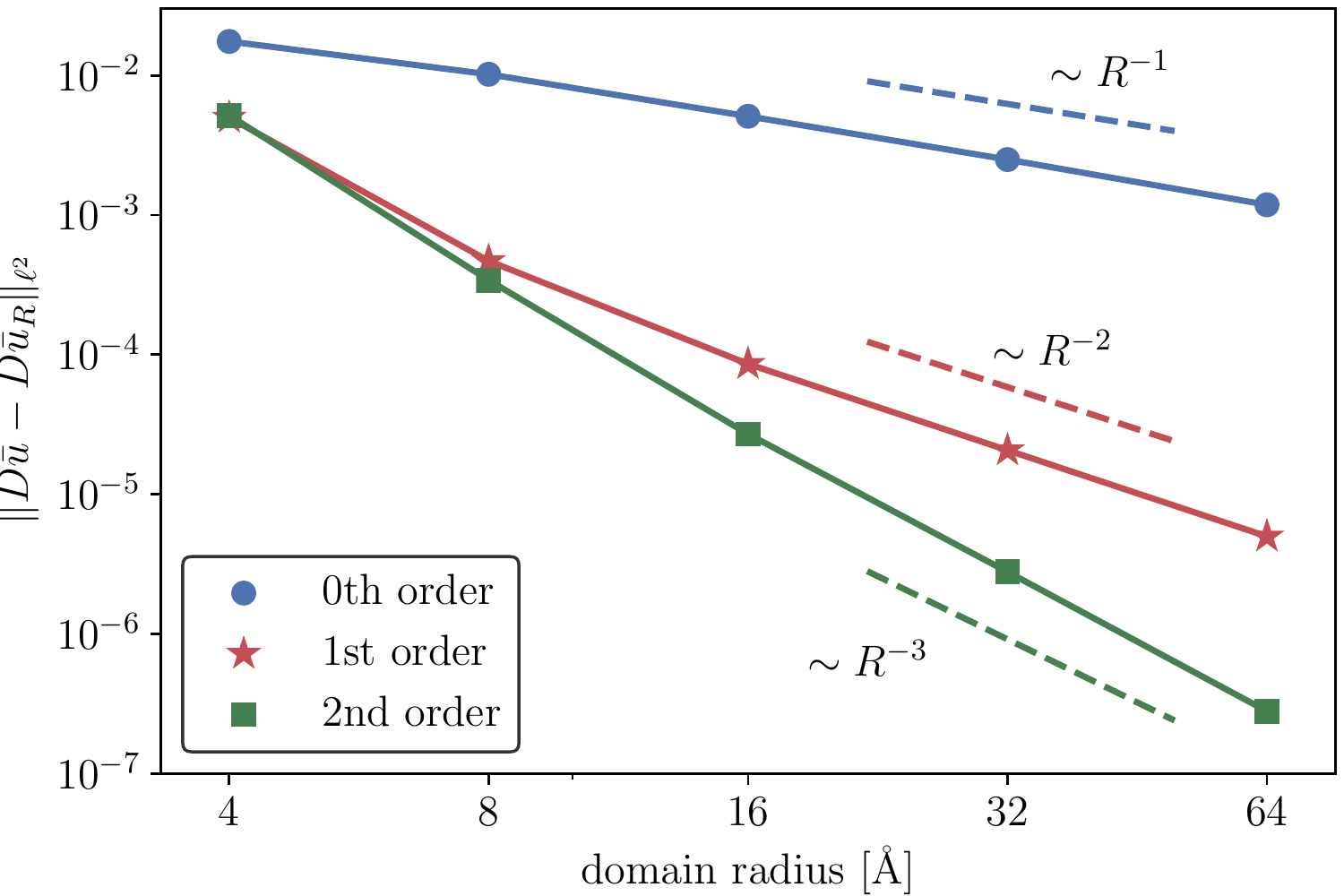}
\end{minipage}
\caption[caption]{Left: Decay of $|\DRc\bar{u}|$ for a BCC hard core screw
dislocation with standard and improved far-field predictors; cf.~\Cref{sec:numerics_bcc}. Transparent dots denote data points $(|x|, |Du(x)|)$, solid curves their envelopes. The numerically observed improved decay for higher-order predictors is consistent with \Cref{BCC-result}.\\\hspace{\textwidth}Right: Rates of convergence of the supercell approximation \cref{eq:approx-prob} to the BCC hard core screw dislocation, employing the standard as well as higher-order far-field predictors. The improved rates of convergence due to the faster decay of the corrector solutions are consistent with \Cref{approx-error-thm}.}
\label{fig:bcc-hard}
\end{figure}

\section{Conclusion} \label{sec:Conclusion}
We have developed a range of results establishing finer properties of the
elastic far-field generated by a screw dislocation in anti-plane shear
kinematics. Of particular note is the role that crystalline symmetries play in
obtaining either cancellation (screw and square lattice) or simple and explicit
representations of the leading order terms of this  elastic far-field.
As a key application we showed how these results can be exploited to obtain
boundary conditions with significantly improved convergence rates in terms
of computational cell size.

Most importantly, though, the general ideas that we outlined in this paper
set the scene for an in-depth study of the elastic far-field for a large variety of defect types, fully vectorial models, and more general crystalline solids. The resulting derivation of higher-order boundary conditions promises to yield simple, efficient as well as highly accurate new algorithms to simulate crystalline defects.

\section{Proofs}\label{Proofs}
\subsection{Auxiliary results about symmetry}
We prove the main results through a number of lemmas, starting with the following observations about how symmetry simplifies the tensors appearing in the development of the forces. This includes, but is not limited to the tensors $\nabla^2 W(0) \in \R^{2 \times 2} = (\R^2)^{\otimes 2}$, $\nabla^3W(0) \in (\R^2)^{\otimes 3}$, and $\nabla^4 W(0)\in (\R^2)^{\otimes 4}$.

Let $m \in \mathbb{N}$, $A\in (\R^2)^{\otimes m}$ and $B \in \R^{2 \times 2}$
then the tensor $B^{\otimes m} A \in (\R^2)^{\otimes m}$ is, as usual, defined by
\[(B^{\otimes m} A)_{l_1 \dots l_m} := \sum_{k \in \{1,2\}^m} A_{k_1 \dots k_m}\prod_{i=1}^m B_{l_i k_i}.\]
As before let $Q$ be a matrix representing either a rotation by $\pi/2$ (in the case $\Lambda = \Z^2$) or a rotation by $2\pi/3$ (in the case $\Lambda = A_{\rm tri} \Z^2$). That is,
\[Q=\begin{pmatrix}
0 & -1 \\ 1 & 0
\end{pmatrix}\text{ or }\quad Q=\begin{pmatrix}
-\sfrac{1}{2} & -\sfrac{\sqrt{3}}{2} \\ \sfrac{\sqrt{3}}{2} & -\sfrac{1}{2}
\end{pmatrix}.\]
More generally, let $Q \in \R^{2 \times 2}$ with $Q^N=\Id$ for some $N \in \N$, $N \geq 1$, and $Q^T Q = \Id$. Our specific cases are included as $N=4$ and $N=3$. We then define
\[PA= \frac{1}{N}\sum_{M=0}^{N -1} (Q^M)^{\otimes m} A.\]

Consider the standard scalar product for tensors,
\[A : B = \sum_{k_1, \dots, k_m =1}^2 A_{k_1 \dots k_m} B_{k_1 \dots k_m}.\]
Then we have the following lemma.
\begin{lemma}\label{projectorlemma}
$P$ is the orthogonal projector onto the $Q$-invariant tensors \[ \{A \colon Q^{\otimes m} A = A\}. \]
\end{lemma}
\begin{proof}
One readily checks that $Q^{\otimes m} ((Q^M)^{\otimes m} A) = (Q^{M+1})^{\otimes m} A$. Using also $Q^N=\Id$ one immediately obtains $Q^{\otimes m} PA = PA$. Therefore, $P^2 =P$. Since $(Q^M)^T= Q^{-M} = Q^{N-M}$, we also see that $P$ is self-adjoint. Hence, $P$ is an orthogonal projection onto a subspace of $\{A \colon Q^{\otimes m} A = A\}$. But if $Q^{\otimes m} A = A$ then clearly $PA =A$, which concludes the proof.
\end{proof}

\Cref{projectorlemma} will prove highly useful: Explicitly calculating $P$ now allows us to characterise the rotationally invariant tensors.

To simplify that calculation further, we also define the symmetric part by
\[(\sym A)_{l_1 \dots l_m} = \frac{1}{m!}\sum_{\varphi \in S_m} A_{\varphi(l_1)\dots \varphi(l_m)},\]
where $S_m$ is the group of all permutations on $m$ numbers. For all $A$ we define
\[
 P_{\sym}A := P\sym A=\sym PA.
\]
Let us calculate these projections and thus the invariant spaces for the cases we encounter in our proof later.

For a simple notation of three-tensors and four-tensors in the following we will write $E_{ijk} = e_i \otimes e_j \otimes e_k$ and $E_{ijkl} = e_i \otimes e_j \otimes e_k \otimes e_l$ where $\{e_1, e_2\}$ represents the standard base of $\R^2$.

\begin{lemma}\label{tensors-lemma}
\begin{enumerate}[label=(\alph*)]
\item For $m=2$ and $N \geq 3$,
  \[ P_{\sym} A = \sfrac{1}{2} \tr(A) \Id,
    \quad \text{i.e.,} \quad
    \{A \colon Q^{\otimes 2} A = A, \sym A = A\} = \spano \Id
  \]
\item For $m=3$ and $N = 3$,
\begin{align*}
P_{\sym} A
&=\sfrac{1}{4}( E_{111} - 3 \sym E_{122}) (A_{111} - 3 \sym A_{122})\\
 &\quad + \sfrac{1}{4}( E_{222} - 3 \sym E_{112}) (A_{222} - 3 \sym A_{112}), \\
 & \hspace{-2cm} \text{i.e., }  \{A \colon Q^{\otimes 3} A = A, \sym A = A\} = \spano \{ E_{111} - 3 \sym E_{122}, E_{222} - 3 \sym E_{112} \}
\end{align*}
\item For $m=4$ and $N=3$,
\begin{align*}
&( P_{\sym} A)_{abcd}
=\sfrac{1}{8}\Big( \delta_{ab}\delta_{cd} +\delta_{ac}\delta_{bd} +\delta_{ad}\delta_{bc}\Big) (A_{1111} + 2 \sym A_{1122} + A_{2222}), \\
\hspace{-1cm} \text{i.e., } &
  \{A \colon Q^{\otimes 4} A = A, \sym A = A\} = \spano \big\{E_{1111} + E_{2222} + 2\sym{E_{1122}} \big\}
\end{align*}
\end{enumerate}
\end{lemma}

\begin{proof}
(a)\ We have $(Q \otimes Q) A = A$ if and only if $Q A Q^T = A$. For symmetric $A$, we can diagonalize $A = R D R^T$ with some rotation $R$ and a diagonal matrix $D$. But then $Q A Q^T = A$ is equivalent to $Q D Q^T = D$. This is the case precisely if $D =c\Id$ or $Q \in \{\pm \Id\}$. Since we excluded the latter option we find $\{ A\colon (Q \otimes Q) A = A\} = \Id \R$ as claimed.

\noindent(b)\ This statement is more involved and notably depends on $N$. Therefore a general argument as in (a) cannot work. One way of obtaining the result is to calculate the projector explicitly. By linearity, it suffices to consider $A= \sigma \otimes \rho \otimes \tau$. In this case, $Q^{\otimes m} A= Q\sigma \otimes Q\rho \otimes Q\tau$. We get
\begin{align*}
3(P (\sigma \otimes \rho \otimes \tau))_{111} 
&= \sigma_1 \rho_1 \tau_1 +(-\sfrac{1}{2}\sigma_1-\sfrac{\sqrt{3}}{2}\sigma_2)(-\sfrac{1}{2}\rho_1-\sfrac{\sqrt{3}}{2}\rho_2)(-\sfrac{1}{2}\tau_1-\sfrac{\sqrt{3}}{2}\tau_2)\\
&\quad +(-\sfrac{1}{2}\sigma_1+\sfrac{\sqrt{3}}{2}\sigma_2)(-\sfrac{1}{2}\rho_1+\sfrac{\sqrt{3}}{2}\rho_2)(-\sfrac{1}{2}\tau_1+\sfrac{\sqrt{3}}{2}\tau_2)\\
&= \sfrac{3}{4}\big(\sigma_1 \rho_1 \tau_1 -\sigma_1 \rho_2 \tau_2- \sigma_2 \rho_1 \tau_2 - \sigma_2 \rho_2 \tau_1\big), \\
3(P (\sigma \otimes \rho \otimes \tau))_{222} 
&= \sigma_2 \rho_2 \tau_2 +(\sfrac{\sqrt{3}}{2}\sigma_1-\sfrac{1}{2}\sigma_2)(\sfrac{\sqrt{3}}{2}\rho_1-\sfrac{1}{2}\rho_2)(\sfrac{\sqrt{3}}{2}\tau_1-\sfrac{1}{2}\tau_2)\\
&\quad +(-\sfrac{\sqrt{3}}{2}\sigma_1-\sfrac{1}{2}\sigma_2)(-\sfrac{\sqrt{3}}{2}\rho_1-\sfrac{1}{2}\rho_2)(-\sfrac{\sqrt{3}}{2}\tau_1-\sfrac{1}{2}\tau_2)\\
&= \sfrac{3}{4}\big(\sigma_2 \rho_2 \tau_2 -\sigma_1 \rho_1 \tau_2- \sigma_1 \rho_2 \tau_1 - \sigma_2 \rho_1 \tau_1\big),  \\
3(P (\sigma \otimes \rho \otimes \tau))_{112}  
&= \sigma_1 \rho_1 \tau_2 +(-\sfrac{1}{2}\sigma_1-\sfrac{\sqrt{3}}{2}\sigma_2)(-\sfrac{1}{2}\rho_1-\sfrac{\sqrt{3}}{2}\rho_2)(\sfrac{\sqrt{3}}{2}\tau_1-\sfrac{1}{2}\tau_2)\\
&\quad +(-\sfrac{1}{2}\sigma_1+\sfrac{\sqrt{3}}{2}\sigma_2)(-\sfrac{1}{2}\rho_1+\sfrac{\sqrt{3}}{2}\rho_2)(-\sfrac{\sqrt{3}}{2}\tau_1-\sfrac{1}{2}\tau_2)\\
&=\sfrac{3}{4}\big(-\sigma_2 \rho_2 \tau_2 +\sigma_1 \rho_1 \tau_2+ \sigma_1 \rho_2 \tau_1 +\sigma_2 \rho_1 \tau_1\big), \qquad \text{and} \\
3(P (\sigma \otimes \rho \otimes \tau))_{122}   
&= \sigma_1 \rho_2 \tau_2 +(-\sfrac{1}{2}\sigma_1-\sfrac{\sqrt{3}}{2}\sigma_2)(\sfrac{\sqrt{3}}{2}\rho_1-\sfrac{1}{2}\rho_2)(\sfrac{\sqrt{3}}{2}\tau_1-\sfrac{1}{2}\tau_2)\\
&\quad +(-\sfrac{1}{2}\sigma_1+\sfrac{\sqrt{3}}{2}\sigma_2)(-\sfrac{\sqrt{3}}{2}\rho_1-\sfrac{1}{2}\rho_2)(-\sfrac{\sqrt{3}}{2}\tau_1-\sfrac{1}{2}\tau_2)\\
&=\sfrac{3}{4}\big(-\sigma_1 \rho_1 \tau_1 +\sigma_1 \rho_2 \tau_2+ \sigma_2 \rho_1 \tau_2 +\sigma_2 \rho_2 \tau_1\big).
\end{align*}
This concludes (b).

\noindent(c)\ Again, this statement depends on $N$, so we will calculate the projector explicitly. Similar as before, it suffices to consider $A= \pi \otimes \sigma \otimes \rho \otimes \tau$. We find
\begin{align*}
3(P A)_{1111}&= \pi_1 \sigma_1 \rho_1 \tau_1 \\
&\quad +(-\sfrac{1}{2}\pi_1-\sfrac{\sqrt{3}}{2}\pi_2)(-\sfrac{1}{2}\sigma_1-\sfrac{\sqrt{3}}{2}\sigma_2)(-\sfrac{1}{2}\rho_1-\sfrac{\sqrt{3}}{2}\rho_2)(-\sfrac{1}{2}\tau_1-\sfrac{\sqrt{3}}{2}\tau_2)\\
&\quad +(-\sfrac{1}{2}\pi_1+\sfrac{\sqrt{3}}{2}\pi_2)(-\sfrac{1}{2}\sigma_1+\sfrac{\sqrt{3}}{2}\sigma_2)(-\sfrac{1}{2}\rho_1+\sfrac{\sqrt{3}}{2}\rho_2)(-\sfrac{1}{2}\tau_1+\sfrac{\sqrt{3}}{2}\tau_2)\\
&= \sfrac{9}{8}(\pi_1 \sigma_1 \rho_1 \tau_1+\pi_2 \sigma_2 \rho_2 \tau_2)+ \sfrac{3}{8}\big(\pi_1 \sigma_1 \rho_2 \tau_2+\pi_1 \sigma_2 \rho_1 \tau_2\\
&\quad+\pi_1 \sigma_2 \rho_2 \tau_1+\pi_2 \sigma_1 \rho_1 \tau_2+\pi_2 \sigma_1 \rho_2 \tau_1+\pi_2 \sigma_2 \rho_1 \tau_1\big) \qquad \text{and} \\
3(P A)_{2222}&= \pi_2 \sigma_2 \rho_2 \tau_2 \\
&\quad +(\sfrac{\sqrt{3}}{2}\pi_1-\sfrac{1}{2}\pi_2)(\sfrac{\sqrt{3}}{2}\sigma_1-\sfrac{1}{2}\sigma_2)(\sfrac{\sqrt{3}}{2}\rho_1-\sfrac{1}{2}\rho_2)(\sfrac{\sqrt{3}}{2}\tau_1-\sfrac{1}{2}\tau_2)\\
&\quad +(-\sfrac{\sqrt{3}}{2}\pi_1-\sfrac{1}{2}\pi_2)(-\sfrac{\sqrt{3}}{2}\sigma_1-\sfrac{1}{2}\sigma_2)(-\sfrac{\sqrt{3}}{2}\rho_1-\sfrac{1}{2}\rho_2)(-\sfrac{\sqrt{3}}{2}\tau_1-\sfrac{1}{2}\tau_2)\\
&= \sfrac{9}{8}(\pi_1 \sigma_1 \rho_1 \tau_1+\pi_2 \sigma_2 \rho_2 \tau_2)+ \sfrac{3}{8}\big(\pi_1 \sigma_1 \rho_2 \tau_2+\pi_1 \sigma_2 \rho_1 \tau_2\\
&\quad+\pi_1 \sigma_2 \rho_2 \tau_1+\pi_2 \sigma_1 \rho_1 \tau_2+\pi_2 \sigma_1 \rho_2 \tau_1+\pi_2 \sigma_2 \rho_1 \tau_1\big).
\end{align*}
By interchanging $\pi,\sigma,\rho,\tau$, the even mixed terms can be reduced to calculating just
\begin{align*}
3(P A)_{1122}&= \pi_1 \sigma_1 \rho_2 \tau_2 \\
&\quad +(-\sfrac{1}{2}\pi_1-\sfrac{\sqrt{3}}{2}\pi_2)(-\sfrac{1}{2}\sigma_1-\sfrac{\sqrt{3}}{2}\sigma_2)(\sfrac{\sqrt{3}}{2}\rho_1-\sfrac{1}{2}\rho_2)(\sfrac{\sqrt{3}}{2}\tau_1-\sfrac{1}{2}\tau_2)\\
&\quad +(-\sfrac{1}{2}\pi_1+\sfrac{\sqrt{3}}{2}\pi_2)(-\sfrac{1}{2}\sigma_1+\sfrac{\sqrt{3}}{2}\sigma_2)(-\sfrac{\sqrt{3}}{2}\rho_1-\sfrac{1}{2}\rho_2)(-\sfrac{\sqrt{3}}{2}\tau_1-\sfrac{1}{2}\tau_2)\\
&= \sfrac{9}{8}(\pi_1 \sigma_1 \rho_2 \tau_2+\pi_2 \sigma_2 \rho_1 \tau_1)+ \sfrac{3}{8}\big(\pi_1 \sigma_1 \rho_1 \tau_1+\pi_2 \sigma_2 \rho_2 \tau_2\\
&\quad-\pi_1 \sigma_2 \rho_2 \tau_1-\pi_2 \sigma_1 \rho_1 \tau_2-\pi_2 \sigma_1 \rho_2 \tau_1-\pi_1 \sigma_2 \rho_1 \tau_2\big).
\end{align*}
For the symmetric part, these formulae simplify to
\begin{align*}
(P \sym A)_{1111}&= (P \sym A)_{2222}\\
&= 3(P \sym A)_{1122}\\
&= \sfrac{3}{8}(\pi_1 \sigma_1 \rho_1 \tau_1+\pi_2 \sigma_2 \rho_2 \tau_2)+ \sfrac{6}{8} \sym (\pi \otimes \sigma \otimes \rho \otimes \tau)_{1122},
\end{align*}
Furthermore,
\begin{align*}
3(P A)_{1112}&= \pi_1 \sigma_1 \rho_1 \tau_2 \\
&\quad +(-\sfrac{1}{2}\pi_1-\sfrac{\sqrt{3}}{2}\pi_2)(-\sfrac{1}{2}\sigma_1-\sfrac{\sqrt{3}}{2}\sigma_2)(-\sfrac{1}{2}\rho_1-\sfrac{\sqrt{3}}{2}\rho_2)(\sfrac{\sqrt{3}}{2}\tau_1-\sfrac{1}{2}\tau_2)\\
&\quad +(-\sfrac{1}{2}\pi_1+\sfrac{\sqrt{3}}{2}\pi_2)(-\sfrac{1}{2}\sigma_1+\sfrac{\sqrt{3}}{2}\sigma_2)(-\sfrac{1}{2}\rho_1+\sfrac{\sqrt{3}}{2}\rho_2)(-\sfrac{\sqrt{3}}{2}\tau_1-\sfrac{1}{2}\tau_2)\\
&= \sfrac{9}{8}(\pi_1 \sigma_1 \rho_1 \tau_2-\pi_2 \sigma_2 \rho_2 \tau_1)+ \sfrac{3}{8}\big(\pi_2 \sigma_2 \rho_1 \tau_2+\pi_2 \sigma_1 \rho_2 \tau_2\\
&\quad+\pi_1 \sigma_2 \rho_2 \tau_2-\pi_1 \sigma_1 \rho_2 \tau_1-\pi_2 \sigma_1 \rho_1 \tau_1-\pi_2 \sigma_1 \rho_1 \tau_1\big),
\end{align*}
which implies $(P \sym A)_{1112}=0$. In the same spirit one finds $(P \sym A)_{1222}=0$.
\end{proof}

Additionally, the result for $m=N=3$ simplifies further if we add line reflection symmetry. As in \Cref{BCC}, let $a= (\frac{\sqrt{3}}{2}, \frac{1}{2} )^T$ and
\[S = a \otimes a - a^\bot \otimes a^\bot=\begin{pmatrix}
\frac{1}{2} &\frac{\sqrt{3}}{2} \\ \frac{\sqrt{3}}{2} & - \frac{1}{2}
\end{pmatrix}. \]

\begin{lemma}\label{tensors-lemma2}
For $m=3$ and $N =3$ one has
\[
  \{A \colon Q^{\otimes 3} A = A, \sym A = A, S^{\otimes 3} A = - A \} = \spano \{ E_{111} - 3 \sym E_{122}\}.
  \]
\end{lemma}
\begin{proof}
Let $A$ be a tensor with $Q^{\otimes 3} A = A$, $\sym A = A$, and $S^{\otimes 3} A = - A$. According to \Cref{tensors-lemma}, $A = c_1( E_{111} - 3 \sym E_{122}) + c_2 ( E_{222} - 3 \sym E_{112})$. Additionally, $S^{\otimes 3} A = - A$ implies $A[Sa,Sa,Sa] = - A[a,a,a]$. But with $Sa=a$ we have $A[a,a,a]=0$; that is,
\[0= c_1 (\sfrac{3\sqrt{3}}{8} - 3\sfrac{\sqrt{3}}{8}) + c_2 ( \sfrac{1}{8} - 3\sfrac{3}{8}),\]
which implies $c_2 = 0$. With the same calculation one also sees the reverse, i.e., that $E_{222} - 3 \sym E_{112}$ does indeed satisfy the reflection symmetry.
\end{proof}

Among other applications later on in the analysis, \Cref{tensors-lemma,tensors-lemma2} can be used for the following two corollaries. As a first corollary, we recover a classical result about isotropic linear elasticity (compare, e.g., \cite{LL59} for the analogous three-dimensional case).

\begin{corollary}\label{D2W-cor}
In the setting of \Cref{model}, for $W$ given by \cref{CB-W}, one finds $\nabla^2W(0) = c_{\rm lin} \Id$, for some $c_{\rm lin} > 0$, and therefore
 \[
  -\divo (\nabla^2 W(0)[\nabla u]) =-c_{\rm lin} \Delta u
 \]
\end{corollary}
\begin{proof}
According to \cref{CB-W}  $W(F) = \frac{1}{\det A_{\La}}V(F\cdot\Rc)$, hence we have
\[
\nabla^2W(0) = \frac{1}{\det A_{\La}}\sum_{\rho,\sigma\in\Rc} \nabla^2 V(0)_{\rho\sigma}\rho\otimes\sigma.
\]
We further notice that due to the rotational symmetry of $\Rc$ and $V$, \cref{eq:rotsym}, we have $\nabla^2V(0)_{\rho \sigma}= \nabla^2V(0)_{Q\rho Q\sigma}$, hence we can equivalently write
\[
 \nabla^2 W(0) = \frac{1}{\det A_{\La}}\sum_{\rho,\sigma\in\Rc} \nabla^2 V(0)_{\rho \sigma}Q\rho\otimes Q\sigma.
\]
In particular,
\[
 \nabla^2 W(0) \in \{A \in \R^{2\times 2} \,:\,(Q\otimes Q)A = A\}.
\]
It is also clear that $\nabla^2W(0)$ is symmetric, thus
\[
 P_{\sym}\nabla^2W(0) = \nabla^2W(0)
\]
and so we invoke \Cref{tensors-lemma} to conclude that
\[
 \nabla^2W(0) = \left(\frac{1}{2 \det A_{\La}}\sum_{\rho,\sigma\in\Rc} \nabla^2V(0)_{\rho \sigma} \rho \cdot \sigma\right) \Id =: c_{\rm lin} \Id.
\]
Since lattice stability implies Legendre-Hadamard stability of the Cauchy-Born limit \cite{HO2012, braun16static}, it follows that $c_{\rm lin}>0$.
\end{proof}

As a second corollary, we can even identify the lowest order nonlinearity.

\begin{corollary}\label{D3W-cor}
In the setting of \Cref{model} assuming additionally the line reflection symmetry \cref{eq:linereflectionsym}, for $W$ given by \cref{CB-W}, one finds $\nabla^3W(0) = c_{\rm quad}( E_{111} - 3 \sym E_{122})$,
for some $c_{\rm quad} \in \R$, and therefore
 \[
  \divo (\nabla^3 W(0)[\nabla u, \nabla v]) =c_{\rm quad} \bigg(\begin{pmatrix}
\partial_{11} v - \partial_{22} v \\ -2 \partial_{12} v
\end{pmatrix} \cdot \nabla u  +  \begin{pmatrix}
\partial_{11} u - \partial_{22} u \\ -2 \partial_{12} u
\end{pmatrix}\cdot\nabla v \bigg).
 \]
\end{corollary}
\begin{proof}
As $W(F) = \frac{1}{\det A_{\La}}V(F\cdot\Rc)$, we have
\[
\nabla^3W(0) = \frac{1}{\det A_{\La}}\sum_{\rho,\sigma, \tau\in\Rc} \nabla^3V(0)_{\rho\sigma \tau}\rho\otimes\sigma \otimes \tau.
\]
Further, due to the rotational symmetry of $\Rc$ and $V$, \cref{eq:rotsym}, we have $\nabla^3V(0)_{\rho \sigma \tau}=\nabla^3V(0)_{Q\rho Q\sigma Q\tau}$, hence we can equivalently write
\[
 \nabla^3W(0) = \frac{1}{\det A_{\La}}\sum_{\rho,\sigma, \tau \in\Rc} \nabla^3V(0)_{\rho \sigma \tau}Q\rho\otimes Q\sigma \otimes Q\tau.
\]
Furthermore, the line reflection symmetry \cref{eq:linereflectionsym} implies $\nabla^3V(0)_{\rho \sigma \tau}= - \nabla^3 V(0)_{S\rho S\sigma S\tau}$, which translates to
\[
 \nabla^3 W(0) = - \frac{1}{\det A_{\La}}\sum_{\rho,\sigma, \tau \in\Rc} \nabla^3 V(0)_{\rho \sigma \tau}S\rho\otimes S\sigma \otimes S\tau.
\]
Combining these observations, we find
\[
 \nabla^3W(0) \in \big\{A \colon A= \sym A, Q^{\otimes 3}A = A, S^{\otimes 3}A = -A \big\},
\]
and invoking \Cref{tensors-lemma}, we therefore deduce that
\begin{align*}
 \nabla^3W(0) &= \Big(\frac{1}{4 \det A_{\La}}\sum_{\rho,\sigma, \tau\in\Rc} \nabla^3V(0)_{\rho \sigma \tau} (\rho_1 \sigma_1 \tau_1 - \rho_1 \sigma_2 \tau_2 - \rho_2 \sigma_1 \tau_2 - \rho_2 \sigma_2 \tau_1 )\Big)\\
 &\quad \quad \cdot(E_{111} - 3 \sym E_{122})\\
 &=: c_{\rm quad} (E_{111} - 3 \sym E_{122}).
\end{align*}
Finally, the identity
\[\divo ((E_{111} - 3 \sym E_{122})[\nabla u, \nabla v]) = \begin{pmatrix}
\partial_{11} v - \partial_{22} v \\ -2 \partial_{12} v
\end{pmatrix} \cdot \nabla u  +  \begin{pmatrix}
\partial_{11} u - \partial_{22} u \\ -2 \partial_{12} u
\end{pmatrix}\cdot\nabla v.\]
completes the proof.
\end{proof}

\subsection{Decay of the linear residual}
\label{sec:decay_ref}
As discussed in the sketch of the proof, see \cref{eq:linearresidual}, the crucial object is the \emph{linear residual}
\begin{equation*}
f_u = - \divRc(\nabla^2V(0)[\DRc u]).
\end{equation*}
We now establish how crystalline symmetries lead to a faster decay of $f_u$ as
would be expected from linearised elasticity in general.

\begin{theorem} \label{decaylinearresidual}
\begin{enumerate}[label=(\alph*)]
\item In the setting of \Cref{thm:highsym}, on a square lattice, we have
\[\lvert f_{\bar{u}}(x) \rvert \lesssim \lvert x \rvert^{-4}\]
for sufficiently large $\lvert x \rvert$.
\item In the setting of \Cref{thm:highsym} on a triangular lattice, we have
\begin{align*}
\lvert f_{\bar{u}}(x) \rvert &\lesssim \lvert x \rvert^{-6} \log^2 \lvert x \rvert + \lvert x \rvert^{-3} \lvert \DRc \bar{u} \rvert + \lvert x \rvert^{-2} \lvert \DRc^2 \bar{u} \rvert\\
&\lesssim \lvert x \rvert^{-5} \log \lvert x \rvert
\end{align*}
for sufficiently large $\lvert x \rvert$.
\item In the setting of \Cref{BCC-result}, we have
\[\lvert f_{\bar{u}}(x) \rvert \lesssim \lvert x \rvert^{-3}\]
for sufficiently large $\lvert x \rvert$. But, writing $\bar{u}=u_1 + u_2 + \bar{u}_{\rm rem}$ with $u_1$ and $u_2$ given by \cref{u1form} and \cref{u2form}, we have
\begin{align*}
\lvert f_{\bar{u}_{\rm rem}}(x) \rvert &\lesssim \lvert \DRc^2 \bar{u}_{\rm rem} \rvert \lvert \DRc \bar{u}_{\rm rem} \rvert +  \lvert x \rvert^{-2} \lvert \DRc \bar{u}_{\rm rem} \rvert +\lvert x \rvert^{-1} \lvert \DRc^2 \bar{u}_{\rm rem} \rvert + \lvert x \rvert^{-5}\\
&\lesssim \lvert x \rvert^{-4} \log \lvert x \rvert.
\end{align*}
for sufficiently large $\lvert x \rvert$.
\end{enumerate}
\end{theorem}
\begin{remark}
\Cref{decaylinearresidual} improves on the residual decay estimate $\lvert x \rvert^{-3}$ obtained in \cite{EOS2016} in all three cases we consider. This can be used to gain better estimates on $\bar{u}$ or $\bar{u}_{\rm rem}$ which in turn improves the rates here. Iteratively, we will see that the terms involving $\bar{u}$ or $\bar{u}_{\rm rem}$ in all of the above estimates turn out to be negligible.
\end{remark}
\begin{proof}
Recall that $\bar{u}$ is a critical point of the energy difference, satisfying the equilibrium equation
\begin{equation} \label{eq:forceequilibrium}
- \divRc(\nabla V(\DRc \hat{u}+\DRc \bar{u}))=0.
\end{equation}
To obtain an estimate on $f_{\bar{u}}(x)$ we first linearise by Taylor expansion of $V$ around $0$ and then connect to CLE by Taylor expansion of $\DRc \hat{u}$ around $x$. Note that $\hat{u}$ is not smooth at the branch cut $\Gamma$ and $D_\rho \hat{u}$ is not close to $\nabla \hat{u} \cdot \rho$ there either. But this is not a problem as the jump of $\hat{u}$ is equal to the periodicity $p$ (or $-p$) of $V$ and $\nabla \hat{u} \in C^\infty (\R^2\backslash \{0\})$.
Therefore, one can always substitute $\hat{u}(x)$ by $\hat{u}(x) \pm p$ where necessary. We will use this implicitly in the following arguments.

Taylor expanding $V$ around $0$ and ordering by order of decay gives
\[0 = f_{\bar{u}} + I_2 + I_3 + I_4 + I_5 + I_{\rm rem}\]
where
\begin{align*}
I_2 &= - \divRc(\nabla^2V(0)[\DRc \hat{u}]),\\
I_3 &= - \frac{1}{2}\divRc(\nabla^3V(0)[\DRc \hat{u},\DRc \hat{u}]),\\
I_4 &= - \divRc(\nabla^3V(0)[\DRc \hat{u},\DRc \bar{u}]) - \frac{1}{6}\divRc(\nabla^4V(0)[\DRc \hat{u},\DRc \hat{u},\DRc \hat{u}]),\\
I_5 &= - \frac{1}{2}\divRc(\nabla^3V(0)[\DRc \bar{u},\DRc \bar{u}]) - \frac{1}{2}\divRc(\nabla^4V(0)[\DRc \hat{u},\DRc \hat{u},\DRc \bar{u}])\\
&\qquad  - \frac{1}{24} \divRc(\nabla^5V(0)[\DRc \hat{u}]^4),
\end{align*}
and the remainder satisfies
\begin{equation} \label{eq:Irem}
 I_{\rm rem} \rvert \leq \rvert x \lvert^{-6} \log^2 \lvert x \rvert
\end{equation}
due to the already known decay estimates on $\bar{u}$ from \Cref{EOSdecay} and the explicit rates for $\hat{u}$:
\[
  \lvert \DRc^j \hat{u} \rvert \leq \lvert x \rvert^{-j}
  \quad \text{and} \quad
  \lvert \DRc^j \bar{u} \rvert \leq \lvert x \rvert^{-j-1} \log \lvert x \rvert
  \qquad \text{for } j \geq 1.
\]

{\it Estimate for $I_2$: } The term $I_2$  depends only on $\hat{u}$. We can expand $\hat{u}$
\begin{align*}
I_2 &= \sum_{\rho, \sigma \in \Rc} \nabla^2V(0)_{\rho \sigma} D_{-\sigma} D_{\rho} \hat{u}(x)\\
&= \sum_{\rho, \sigma \in \Rc} \nabla^2V(0)_{\rho \sigma}\big(\hat{u}(x+\rho-\sigma)+  \hat{u}(x)-\hat{u}(x + \rho)-\hat{u}(x-\sigma)\big)\\
&= J_2 + J_3 +J_4 + J_5 + O(\lvert x \rvert^{-6}),
\end{align*}
where
\begin{align*}
J_2 &= -\sum_{\rho, \sigma \in \Rc} \nabla^2V(0)_{\rho \sigma} \nabla^2\hat{u}(x)[\rho,\sigma]\\
J_3 &= \frac{1}{2} \sum_{\rho, \sigma \in \Rc} \nabla^2V(0)_{\rho \sigma} \nabla^3\hat{u}(x)\big([\rho,\sigma,\sigma]-[\rho,\rho,\sigma] \big) \\
J_4 &= \frac{1}{12} \sum_{\rho, \sigma \in \Rc} \nabla^2V(0)_{\rho \sigma} \nabla^4\hat{u}(x)\big(-2[\rho,\sigma,\sigma,\sigma]+3[\rho,\rho,\sigma,\sigma]-2[\rho,\rho,\rho,\sigma] \big)\\
J_5 &= \frac{1}{24} \sum_{\rho, \sigma \in \Rc} \nabla^2V(0)_{\rho \sigma} \nabla^5\hat{u}(x)\big([\rho,\sigma,\sigma,\sigma,\sigma]-2[\rho,\rho,\sigma,\sigma,\sigma]\\
&\qquad \qquad \qquad + 2[\rho,\rho,\rho,\sigma,\sigma] - [\rho,\rho,\rho,\rho,\sigma]\big)\\
\end{align*}
Using the symmetry in $\rho$ and $\sigma$ it follows that $J_3=J_5=0$. By \Cref{tensors-lemma},
\begin{align}
J_2 &= -\sum_{\rho, \sigma \in \Rc} \nabla^2V(0)_{\rho \sigma} \nabla^2\hat{u}(x)[\rho,\sigma] \label{eq:J2}\\
&= \big(-\sum_{\rho, \sigma \in \Rc} \nabla^2V(0)_{\rho \sigma} \rho \otimes \sigma \big) \colon \nabla^2\hat{u}(x) \nonumber  \\
&= \big(-\frac{1}{2}\sum_{\rho, \sigma \in \Rc} \nabla^2V(0)_{\rho \sigma} \rho \cdot \sigma \big) \Delta \hat{u}(x) \nonumber.
\end{align}
Hence, $J_2=0$. Thus we conclude so far that $I_2 = J_4 + O(\lvert x \rvert^{-6})$.  To proceed, we now distinguish the specific cases we consider.

{\it Proof of (a): } Due to mirror reflection symmetry we have $\nabla^3V(0)=0$ and $\nabla^5V(0)=0$, hence $I_3=0$, $\lvert I_4 \rvert \lesssim \lvert x \rvert^{-4}$ and $|I_2| \lesssim |J_4| + \lvert x \rvert^{-6} \lesssim \lvert x \rvert^{-4}$. We therefore obtain $\lvert f_{\bar{u}} \rvert \lesssim \lvert x \rvert^{-4}$ which concludes the proof of (a).

{\it Estimates for $J_4, I_4$ for cases (b, c): }
We use \Cref{tensors-lemma} to calculate
\begin{align}
J_4 &= \frac{1}{12} \sum_{\rho, \sigma \in \Rc} \nabla^2V(0)_{\rho \sigma} \nabla^4\hat{u}(x)\big(-2[\rho,\sigma,\sigma,\sigma]+3[\rho,\rho,\sigma,\sigma]-2[\rho,\rho,\rho,\sigma] \big) \label{eq:J4} \\
&= P_{\sym}\Big(\frac{1}{12} \sum_{\rho, \sigma \in \Rc} \nabla^2V(0)_{\rho \sigma} \big(-2\rho\otimes\sigma\otimes\sigma\otimes\sigma+3\rho\otimes\rho\otimes\sigma\otimes\sigma \nonumber\\
&\qquad\qquad-2\rho\otimes\rho\otimes\rho\otimes\sigma \big) \Big) \colon \nabla^4\hat{u}(x) \nonumber\\
&= \Big( \frac{1}{32} \sum_{\rho, \sigma \in \Rc} \nabla^2V(0)_{\rho \sigma} \big( 2 (\rho \cdot \sigma)^2+ \lvert \rho \rvert^2 \lvert \sigma \rvert^2 -2\rho\cdot\sigma (\lvert \rho \rvert^2 + \lvert \sigma \rvert^2 \big)\Big) \Delta^2\hat{u}(x) \nonumber
\end{align}
As $\Delta^2 \hat{u}=0$, we find $J_4=0$ and hence obtain $\lvert I_2 \rvert \lesssim \lvert x \rvert^{-6}$.

Next, we consider
\begin{align*}
I_4 &= - \frac{1}{6}\divRc(\nabla^4V(0)[\DRc \hat{u},\DRc \hat{u},\DRc \hat{u}])\\
&= \frac{1}{6} \sum_{\pi,\rho,\sigma,\tau} \nabla^4V(0)_{\pi\rho\sigma\tau} D_{-\tau}( D_{\pi} \hat{u} D_{\rho} \hat{u}  D_{\sigma} \hat{u})\\
&= - \frac{1}{2} \sum_{\pi,\rho,\sigma,\tau} \nabla^4V(0)_{\pi\rho\sigma\tau} \nabla^2 \hat{u}[\pi, \tau] \nabla \hat{u}\cdot \rho  \nabla \hat{u}\cdot \sigma\\
&\qquad - \frac{1}{6} \sum_{\pi,\rho,\sigma,\tau} \nabla^4V(0)_{\pi\rho\sigma\tau} \nabla^3 \hat{u}[\pi, \pi, \tau] \nabla \hat{u}\cdot \rho  \nabla \hat{u}\cdot \sigma\\
&\qquad + \frac{1}{6} \sum_{\pi,\rho,\sigma,\tau} \nabla^4V(0)_{\pi\rho\sigma\tau} \nabla^3 \hat{u}[\pi, \tau, \tau] \nabla \hat{u}\cdot \rho  \nabla \hat{u}\cdot \sigma\\
&\qquad - \frac{1}{2} \sum_{\pi,\rho,\sigma,\tau} \nabla^4V(0)_{\pi\rho\sigma\tau} \nabla^2 \hat{u}[\pi, \tau] \nabla^2 \hat{u}[\rho, \rho]  \nabla \hat{u}\cdot \sigma\\
&\qquad + \frac{1}{2} \sum_{\pi,\rho,\sigma,\tau} \nabla^4V(0)_{\pi\rho\sigma\tau} \nabla^2 \hat{u}[\pi, \tau] \nabla^2 \hat{u}[\rho, \tau]  \nabla \hat{u}\cdot \sigma.
\end{align*}
The second and third terms cancel each other by symmetry in $\pi$ and $\tau$, while the fourth and fifth terms both vanish due to $\nabla^4V(0)_{\pi\rho\sigma\tau} = \nabla^4V(0)_{(-\pi)(-\rho)(-\sigma)(-\tau) }$.

Applying again \Cref{tensors-lemma} we can express the first term as
\begin{align}
I_4 &= - \frac{1}{2} \sum_{\pi,\rho,\sigma,\tau} \nabla^4V(0)_{\pi\rho\sigma\tau} \nabla^2 \hat{u}[\pi, \tau] \nabla \hat{u}\cdot \rho  \nabla \hat{u}\cdot \sigma \label{eq:I4}\\
&= P_{\sym}\Big(- \frac{1}{2} \sum_{\pi,\rho,\sigma,\tau} \nabla^4V(0)_{\pi\rho\sigma\tau} \pi \otimes \tau \otimes \rho \otimes \sigma \Big) \colon ( \nabla^2 \hat{u}\otimes \nabla \hat{u}\otimes  \nabla \hat{u})\nonumber\\
&= \frac{1}{24} \Big(- \frac{1}{2} \sum_{\pi,\rho,\sigma,\tau} \nabla^4V(0)_{\pi\rho\sigma\tau} ((\pi\cdot \tau)(\rho \cdot \sigma)+(\pi\cdot \rho)(\tau \cdot \sigma)+(\pi\cdot \sigma)(\rho \cdot \tau)) \Big) \nonumber\\
&\qquad\cdot( 3\partial_1^2 \hat{u} (\partial_1 \hat{u})^2+3\partial_2^2 \hat{u} (\partial_2 \hat{u})^2 +\partial_1^2 \hat{u} (\partial_2 \hat{u})^2+\partial_2^2 \hat{u} (\partial_1 \hat{u})^2 + 4 \partial_1\partial_2 \hat{u} \partial_1 \hat{u} \partial_2 \hat{u}) \nonumber\\
&= c (\lvert \nabla \hat{u} \rvert^2 \Delta \hat{u} + 2  \nabla^2 \hat{u} [\nabla \hat{u}, \nabla \hat{u}]) \nonumber\\
&= c\big(3\lvert \nabla \hat{u}(x)\rvert^2+2) \Delta \hat{u}(x)-4(1+ \lvert \nabla \hat{u}\rvert^2\big)^{\frac{3}{2}}H \nonumber,
\end{align}
where
\[H=\frac{(1+ (\partial_1 \hat{u})^2)\partial_2^2 \hat{u} + (1+ (\partial_2 \hat{u})^2)\partial_1^2 \hat{u} - 2 \partial_1 \hat{u} \partial_2 \hat{u} \partial_1 \partial_2 \hat{u}}{2(1+ \lvert \nabla \hat{u}\rvert^2)^{\frac{3}{2}}}\]
is the mean curvature of the surface given by $x_3 = \hat{u}(x_1,x_2)$. Since the graph of $\hat{u}$ is a helicoid, i.e., a minimal surface, $H \equiv 0$ and therefore we
 have shown that $I_4=0$.

{\it Proof of (b): }  Due to the mirror symmetry we again obtain
$\nabla^3V(0)=\nabla^5V(0)=0$, hence $I_3=0$. In addition, again due to mirror
symmetry, $I_5$ simplifies to
\[I_5 = - \frac{1}{2}\divRc(\nabla^4V(0)[\DRc \hat{u},\DRc \hat{u},\DRc \bar{u}]).\]
Therefore,
\begin{align*}
\lvert I_5 \rvert &\lesssim \lvert x \rvert^{-3} \lvert \DRc \bar{u} \rvert + \lvert x \rvert^{-2} \lvert \DRc^2 \bar{u} \rvert\\
&\lesssim \lvert x \rvert^{-5} \log \lvert x \rvert,
\end{align*}
Invoking $I_4 = 0$ and $|I_2| \lesssim |x|^{-6}$ from the previous step concludes
the proof of (b).

{\it Proof of (c): } On the BCC lattice, case (c), one typically finds $\nabla^3 V(0)\neq 0$ and $\nabla^5 V(0)\neq 0$. In particular, $I_3$ does not vanish, hence our arguments so far only yield $\lvert f_{\bar{u}} \rvert \leq \lvert x \rvert^{-3}$.

To estimate, $f_{\bar{u}_{\rm rem}}$ we replace $\hat{u}$ with $\hat{u} + u_1 + u_2$ and $\bar{u}$ with $\bar{u}_{\rm rem}$ in the previous steps of the proof. Recall from \cref{eq:u1u2estimates} that  $\lvert \nabla^j u_i \rvert \lesssim \lvert x \rvert^{-i-j}$.

Clearly, \Cref{eq:forceequilibrium} and the Taylor expansion of $V$ including \cref{eq:Irem} still hold. For the estimates let us start with the higher order terms.  We can estimate directly
\[\big\lvert - \divRc(\nabla^3V(0)[\DRc (\hat{u}+u_1 + u_2),\DRc \bar{u}_{\rm rem}]) \big\rvert \leq  \lvert x \rvert^{-2} \lvert \DRc \bar{u}_{\rm rem} \rvert +\lvert x \rvert^{-1} \lvert \DRc^2 \bar{u}_{\rm rem} \rvert.\]
If we also substitute $\hat{u}$ by $\hat{u} + u_1 + u_2$ in \cref{eq:I4}, we find overall that
\begin{align*}
\lvert I_4 \rvert &\lesssim \lvert x \rvert^{-2} \lvert \DRc \bar{u}_{\rm rem} \rvert +\lvert x \rvert^{-1} \lvert \DRc^2 \bar{u}_{\rm rem} \rvert + \lvert x \rvert^{-5}\\
&\lesssim \lvert x \rvert^{-4} \log \lvert x \rvert.
\end{align*}
In the same spirit we estimate
\begin{align*}
\lvert I_5 \rvert &\lesssim \lvert \DRc^2 \bar{u}_{\rm rem} \rvert \lvert \DRc \bar{u}_{\rm rem} \rvert + \lvert x \rvert^{-3} \lvert \DRc \bar{u}_{\rm rem} \rvert +\lvert x \rvert^{-2} \lvert \DRc^2 \bar{u}_{\rm rem} \rvert + \lvert x \rvert^{-5}\\
&\lesssim \lvert x \rvert^{-5} \log^2 \lvert x \rvert.
\end{align*}
The important difference to before are found in $I_2$ and $I_3$. Let us start with $I_2$. As before, we find $J_3 = J_5 =0$. Substituting $\hat{u}$ by $\hat{u}+ u_1 +u_2$ in \cref{eq:J4}, we estimate
\[\lvert J_4 \rvert \lesssim \lvert \Delta^2 (\hat{u}+u_1 + u_2) \rvert = \lvert \Delta^2 (u_1 + u_2) \rvert \lesssim \lvert x \rvert^{-5}.\]
Therefore, $I_2 = J_2 + O(\lvert x \rvert^{-5})$. It is crucial that now $J_2$ does not vanish to be able to cancel out the first terms in the nonlinearity $I_3$. Following \cref{eq:J2}, we have
\begin{align*}
J_2 &= -\sum_{\rho, \sigma \in \Rc} \nabla^2V(0)_{\rho \sigma} \nabla^2(u_1 + u_2)(x)[\rho,\sigma]\\
&= - \det(A_\Lambda) \divo \big( \nabla^2 W(0)[\nabla (u_1 + u_2)] \big)\\
&= - \det(A_\Lambda) c_{\rm lin} \Delta (u_1 + u_2)
\end{align*}

Now let us come to $I_3$. Clearly,
\begin{align*}
I_3 &= - \frac{1}{2}\divRc(\nabla^3V(0)[\DRc \hat{u},\DRc \hat{u}])\\
&\qquad - \divRc(\nabla^3V(0)[\DRc \hat{u},\DRc u_1]) + O(\lvert x \rvert^{-5}).
\end{align*}
Developing the discrete differences as we did previously for $I_2$, we find
\begin{align*}
- &\divRc(\nabla^3V(0)[\DRc \hat{u},\DRc u_1]) \\
&=  \sum_{\rho, \sigma, \tau \in \Rc} \nabla^3V(0)_{\sigma \rho \tau} D_{-\tau}(D_\rho \hat{u}(x) D_\sigma u_1(x))\\
&= - \sum_{\rho, \sigma, \tau \in \Rc} \nabla^3V(0)_{\sigma \rho \tau} (\nabla \hat{u}(x) [\rho] \nabla^2 u_1(x)[\sigma, \tau] + \nabla u_1(x) [\sigma] \nabla^2 \hat{u}(x)[\rho, \tau])\\
&\qquad + O(\lvert x \rvert^{-5})\\
&= - \det A_\Lambda \divo (\nabla^3W(0)[\nabla \hat{u}, \nabla u_1]) + O(\lvert x \rvert^{-5}).
\end{align*}
For the other term we have to take a few more terms into account. For those we again use the fact that $\nabla^3V(0)_{\sigma \rho \tau} = - \nabla^3V(0)_{(-\sigma) (-\rho) (-\tau)}$.
\begin{align*}
- \frac{1}{2}\divRc&(\nabla^3 V(0)[\DRc \hat{u},\DRc \hat{u}]) \\
&=  \frac{1}{2}\sum_{\rho, \sigma, \tau \in \Rc} \nabla^3 V(0)_{\sigma \rho \tau} D_{-\tau}(D_\rho \hat{u}(x) D_\sigma \hat{u}(x))\\
&=  \frac{1}{2}\sum_{\rho, \sigma, \tau \in \Rc} \nabla^3 V(0)_{\sigma \rho \tau} \big( D_{-\tau}D_\rho \hat{u}(x) D_\sigma \hat{u}(x) + D_\rho \hat{u}(x -\tau) D_{-\tau}D_\sigma \hat{u}(x)  )\\
&= \frac{1}{2}\sum_{\rho, \sigma, \tau \in \Rc} \nabla^3 V(0)_{\sigma \rho \tau} \Big( -2 \nabla \hat{u}(x) [\sigma] \nabla^2 \hat{u}(x)[\rho, \tau]\\
&\qquad + \big(\nabla^3 \hat{u}(x)[\rho, \tau, \tau] - \nabla^3 \hat{u}(x)[\rho, \rho, \tau]\big)\nabla \hat{u}(x)[\sigma]  - \nabla^2 \hat{u}(x)[\rho, \tau] \nabla^2 \hat{u}(x)[\sigma, \sigma]\\
&\qquad + \nabla^2 \hat{u}(x)[\rho, \tau] \nabla^2 \hat{u}(x)[\sigma, \tau] \Big)+ O(\lvert x \rvert^{-5})\\
&= -\sum_{\rho, \sigma, \tau \in \Rc} \nabla^3 V(0)_{\sigma \rho \tau} \nabla \hat{u}(x) [\sigma] \nabla^2 \hat{u}(x)[\rho, \tau]+ O(\lvert x \rvert^{-5})\\
&= - \det A_\Lambda \frac{1}{2}\divo (\nabla^3 W(0)[\nabla \hat{u}, \nabla \hat{u}]) + O(\lvert x \rvert^{-5})
\end{align*}
Hence, we can use \Cref{eqcorr1a,eqcorr2a} for $u_1$ and $u_2$ to conclude that $J_2 + I_3 = O(\lvert x \rvert^{-5})$. This concludes the proof.
\end{proof}

\subsection{Proofs of the main theorems}

The connection between the decay of $f_u$ and the decay of $u$ is as follows:

\begin{theorem} \label{structuretheorem}
Let $u \in \Hcc$, and $j \in \{1,2\}$.
\begin{enumerate}[label=(\alph*)]
\item If $\lvert f_u(x) \rvert \lesssim \lvert x \rvert^{-3}$ and $\sum_x f_u =0$, then
for $\lvert x \rvert$ sufficiently large,
\[
\lvert D^j u(x) \rvert \lesssim \lvert x \rvert^{-1-j} \log \lvert x \rvert.
\]
\item If $\lvert f_u(x) \rvert \lesssim \lvert x \rvert^{-4}$,
$\sum_x f_u = 0$, and $\sum_x f_u x =0$, then for $\lvert x \rvert$ sufficiently large,
\[
\lvert D^j u(x) \rvert \lesssim \lvert x \rvert^{-2-j} \log \lvert x \rvert.
\]
\item If $\lvert f_u(x) \rvert \lesssim \lvert x \rvert^{-5}$,
$\sum_x f_u = 0$, $\sum_x f_u x =0$, and $\sum_x f_u x \otimes x \propto
\Id$, then
for $\lvert x \rvert$ sufficiently large,
\[
\lvert D^j u(x) \rvert \lesssim \lvert x \rvert^{-3-j} \log \lvert x \rvert.
\]
\item If the assumptions on the decay rate of $f_u$ in (a), (b), or (c) are slightly stronger, namely $\lvert x \rvert^{-3 - \eps}$, $\lvert x \rvert^{-4 - \eps}$, or $\lvert x \rvert^{-5 - \eps}$ for some $\eps >0$, then the resulting rates for $\DRc^j u$ are true without the logarithmic term, i.e. $\lvert x \rvert^{-1-j}$, $\lvert x \rvert^{-2-j}$, and $\lvert x \rvert^{-3-j}$, respectively.
\end{enumerate}
\end{theorem}
\begin{proof}
Statement (a) is part of the results in \cite{EOS2016}. Its extensions (b), (c), and (d) follow a similar basic strategy. The approach is based on knowledge about the lattice Green's function $G$ as one can write $Du$ as a convolution on the lattice, $Du = f_u \ast_{\Lambda} DG$, that is,
\[Du(x) = \sum_{z \in \Lambda} f_u(z)  DG(x-z).\]
The proof of (b), (c), and (d) is part of a full theory developed in \cite{MainPaper}. All the details as well as further generalisations will be presented there.
\end{proof}

\Cref{structuretheorem} shows that, to prove the main results in \Cref{symsec,BCC}, in addition to the decay of $f_u$ established in \Cref{sec:decay_ref}, we also need to analyse its moments.

\begin{theorem} \label{moments}
In the setting of \Cref{model}. Let $[u] \in \Hcc$ inherit the rotational symmetry \cref{eq:rotsymu} and let $f_u$ denote the resultant linear residual \cref{eq:linearresidual}.  Then we have $\sum_x f_u =0$, $\sum_x f_u x =0$, and $\sum_x f_u x \otimes x \propto \Id$ for some $c \in \R$, provided the sums converge absolutely.
\end{theorem}
\begin{proof}
We begin with $\sum_x f_u =0$. A version of this statement is already needed in \Cref{energysmooth} since it is directly linked with the net-force of the system. \Cref{energysmooth} was established in \cite{EOS2016}. As there was a gap in the proof, namely a proof of the specific claim $\sum_x f_u =0$ in question here, let us give the details in our specific case: Let $\eta$ be a smooth cut-off function with $\eta(x)=1$ for $\lvert x \rvert \leq 1$ and $\eta(x)=0$ for $\lvert x \rvert \geq 2$ and let $\eta_M(x)=\eta \big( \frac{x}{M}\big)$. Then we have
\begin{align*}
\sum_x f_u &= \lim_{M \to \infty} \sum_x f_u \eta_M\\
&=\lim_{M \to \infty} \sum_x \divRc \big(\nabla V(\DRc \hat{u} +\DRc u) -\nabla V(0) - \nabla^2V(0)[\DRc \hat{u} +\DRc u] \big) \eta_M\\
&\qquad  + \lim_{M \to \infty} \sum_x \divRc \nabla^2V(0)[\DRc \hat{u}] \eta_M \\
&=-\lim_{M \to \infty} \sum_x (\nabla V(\DRc \hat{u} +\DRc u)-\nabla V(0) - \nabla^2V(0))[\DRc \hat{u} +\DRc u])[\DRc \eta_M]\\
&\qquad  - \lim_{M \to \infty} \sum_x \nabla^2V(0)[\DRc \hat{u}, \DRc \eta_M] \\
&=: \lim_{M \to \infty} A_M  + B_M.
\end{align*}

Since the support of $\DRc\eta_M$ is contained in
$\{x\colon M-C \leq \lvert x \rvert \leq 2M+C\}$, for some fixed $C>0$, the first term, $A_M$, can be estimated as a remainder in a Taylor expansion by
\begin{align*}
  |A_M|&=
\Big\lvert \sum_x (\nabla V(\DRc \hat{u} +\DRc u)-\nabla V(0) - \nabla^2V(0))[\DRc \hat{u} +\DRc u])[\DRc \eta_M] \Big\rvert \\
&\lesssim \sum_x \lvert \DRc \hat{u} +\DRc u \rvert^2 \lvert \DRc \eta_M \rvert\\
&\lesssim M^2 M^{-3} + \lVert u \rVert_{\Hcc}^2 M^{-1}\\
&\lesssim M^{-1}.
\end{align*}
For the second term, $B_M$, note that $M-C \leq \lvert x \rvert \leq 2M+C$ implies $D\hat{u} = \nabla \hat{u} \cdot \mathcal{R} + O(M^{-2})$ and $D\eta_M = \nabla \eta_M \cdot \mathcal{R} + O(M^{-2})$. Estimating also the ``quadrature error'' (replacing
the sum by an integral) we obtain
\begin{align*}
  B_M &= \frac{1}{\det A_\Lambda}\int_{\R^2}  \nabla^2V(0)[\nabla \hat{u} \cdot \mathcal{R}, \nabla \eta_M \cdot \mathcal{R}]
    + O(M^{-1}) \\
    &= \int_{\R^2 \backslash B_1(0)}  \nabla^2W(0)[\nabla \hat{u}, \nabla \eta_M]
      + O(M^{-1}),
\end{align*}
where we used the fact that $\nabla\eta_M = 0$ on $B_1(0)$ for $M$ sufficiently large.
Applying Gauß's theorem as well as the fact that $\nabla \hat{u}(x)$ is always orthogonal to $\nu$, we obtain
\begin{align*}
  B_M &= \int_{\partial B_1(\hat{x})}  \nabla^2W(0)[\nabla \hat{u}] \cdot \nu \,dS(x)
    + O(M^{-1}) \\
  &= c_{\rm lin} \int_{\partial B_1(\hat{x})}  \nabla \hat{u} \cdot \nu \,dS(x)
    + O(M^{-1}) \\
  &= O(M^{-1}).
\end{align*}
Thus, we have shown that
\[
  \sum_x f_u = \lim_{M \to \infty} (A_M + B_M) = 0.
\]

To prove our claims about the first and second moments, we first show that rotational symmetry of
$\bar{u}$ implies rotational symmetry of $f_u$, i.e., $f_u(L_Q x) = f_u(x)$:
  \begin{align*}
   f_u(L_Q x) &= \sum_{\rho,\sigma\in\Rc} \nabla^2V(0)_{\rho,\sigma}(D_{\sigma} u(L_Q x - \rho) - D_{\sigma}u(L_Q x))\\
   &= \sum_{\rho,\sigma\in\Rc}\nabla^2V(0)_{Q\rho,Q\sigma}(D_{Q\sigma} u(L_Qx - Q\rho) - D_{Q\sigma}u(L_Qx))\\
   &= \sum_{\rho,\sigma\in\Rc}\nabla^2V(0)_{Q\rho,Q\sigma}(D_{Q\sigma} u(L_Q(x - \rho)) - D_{Q\sigma}u(L_Qx))\\
   &= \sum_{\rho,\sigma\in\Rc}\nabla^2V(0)_{Q\rho,Q\sigma}(D_{\sigma} u(x - \rho) - D_{\sigma}u(x))\\
   &= \sum_{\rho,\sigma\in\Rc}\nabla^2V(0)_{\rho,\sigma}(D_{\sigma} u(x - \rho) - D_{\sigma}u(x))\\
   &=f_u(x),
  \end{align*}
  where we have used $Q\Rc = \Rc$, $D_{\sigma}u(x) = D_{Q\sigma}u(L_Qx)$, as well as the rotational symmetry of $V$, \cref{eq:rotsym}. Let $N=3$ for the triangular lattice and $N=4$ for the quadratic lattice, then
\begin{align*}
  \sum_x f_u(x) x &= \sum_x f_u(x) (x-\hat{x})\\
  &= \frac{1}{N}\sum_x \sum_{j=0}^{N-1} f_u(L_Q^j x) Q^j (x-\hat{x})\\
  &= \frac{1}{N}\sum_x f_u(x) \sum_{j=0}^{N-1} Q^j (x-\hat{x})\\
  &=0
\end{align*}
and similarly for the second moment,
\begin{align*}
  \sum_x f_u(x) x \otimes x &=  \sum_x f_u(x) (x-\hat{x}) \otimes (x-\hat{x})  \\
  &=\frac{1}{N}\sum_x \sum_{j=0}^{N-1} f_u(L_Q^j x) (Q^j (x-\hat{x})) \otimes (Q^j (x-\hat{x}))\\
  &= \sum_x f_u(x) P((x-\hat{x}) \otimes (x-\hat{x}))\\
  &= \Id \Big( \frac{1}{2}\sum_x f_u(x) \lvert x-\hat{x} \rvert^2 \Big)
\end{align*}
where we used \Cref{tensors-lemma} in the last step.
\end{proof}

Finally, we can combine all the foregoing results to prove our main theorems.

\begin{proof}[Proof of \Cref{thm:highsym}]
Let us start with the square lattice. According to \Cref{decaylinearresidual}, we have $\lvert f_{\bar{u}} \rvert \lesssim \lvert x \rvert^{-4}$. In particular, $\sum_x f_{\bar{u}}$ and $\sum_x f_{\bar{u}} x$ converge. Due to \Cref{moments}, $\sum_x f_{\bar{u}} =0$ and $\sum_x f_{\bar{u}} x =0$. Hence, by \Cref{structuretheorem}
\[
\lvert D^j \bar{u}(x) \rvert \lesssim \lvert x \rvert^{-2-j} \log \lvert x \rvert.
\]
for $j=1,2$ and $\lvert x \rvert$ large enough.

For the triangular lattice \Cref{decaylinearresidual} gives us $\lvert f_{\bar{u}} \rvert \lesssim \lvert x \rvert^{-5} \log \lvert x \rvert \lesssim \lvert x \rvert^{-4- \eps}$. In particular, $\sum_x f_{\bar{u}}$, $\sum_x f_{\bar{u}}$, and $\sum_x f_{\bar{u}} x \otimes x$ converge. Due to \Cref{moments}, $\sum_x f_{\bar{u}} =0$, $\sum_x f_{\bar{u}} x =0$, and $\sum_x f_{\bar{u}} x \otimes x =c\Id$. At first, by \Cref{structuretheorem} we conclude that
\[
\lvert D^j \bar{u}(x) \rvert \lesssim \lvert x \rvert^{-2-j}.
\]
for $j=1,2$ and $\lvert x \rvert$ large enough. But then \Cref{decaylinearresidual} gives the stronger result $\lvert f_{\bar{u}} \rvert \lesssim \lvert x \rvert^{-6} \log^2 \lvert x \rvert \leq \lvert x \rvert^{-5 -\eps}$, so that by \Cref{structuretheorem} we indeed get
\[
\lvert D\bar{u}(x) \rvert \lesssim \lvert x \rvert^{-3-j}
\]
for $j=1,2$ and $\lvert x \rvert$ large enough.
\end{proof}

\begin{proof}[Proof of \Cref{BCC-result}]
  As in the triangular lattice case we have to argue in several steps. As a starting point \Cref{decaylinearresidual}  shows that $\lvert f_{\bar{u}_{\rm rem}} \rvert \lesssim \lvert x \rvert^{-4} \log \lvert x \rvert \leq \lvert x \rvert^{-3-\eps} $. In particular, $\sum_x f_{\bar{u}_{\rm rem}}$ and $\sum_x f_{\bar{u}_{\rm rem}} x$ converge. Due to \Cref{moments}, $\sum_x f_{\bar{u}_{\rm rem}} =0$ and $\sum_x f_{\bar{u}_{\rm rem}} x =0$. With \Cref{structuretheorem} we find
\[
\lvert D^j \bar{u}_{\rm rem}(x) \rvert \lesssim \lvert x \rvert^{-1-j}
\]
for $j=1,2$, which in turn gives the improved estimate $\lvert f_{\bar{u}_{\rm rem}} \rvert \lesssim \lvert x \rvert^{-4}$ in \Cref{decaylinearresidual}. Going back to \Cref{structuretheorem} we now get
\[
\lvert D^j \bar{u}_{\rm rem}(x) \rvert \lesssim \lvert x \rvert^{-2-j} \log \lvert x \rvert.
\]
Another iteration of \Cref{decaylinearresidual,structuretheorem} improves this to
\[
\lvert D^j \bar{u}_{\rm rem}(x) \rvert \lesssim \lvert x \rvert^{-2-j}.
\]
Finally, by \Cref{decaylinearresidual}, we now find $\lvert f_{\bar{u}_{\rm rem}} \rvert \lesssim \lvert x \rvert^{-5}$. In particular, $\sum_x f_{\bar{u}_{\rm rem}} x \otimes x$ converges as well and due to \Cref{moments}, $\sum_x f_{\bar{u}_{\rm rem}} x \otimes x=c \Id= c' \nabla^2 W(0)$. A last use of \Cref{structuretheorem} gives the desired result,
\[
\lvert D^j \bar{u}_{\rm rem}(x) \rvert \lesssim \lvert x \rvert^{-3-j} \log \lvert x \rvert
\]
for $j=1,2$ and $\lvert x \rvert$ large enough.
\end{proof}

\section*{Acknowledgements}
We thank Tom Hudson and Petr Grigorev for insightful discussions on symmetries of screw dislocations in BCC.

\bibliographystyle{siamplain}
\bibliography{references}

\end{document}